\documentclass[a4paper,11pt]{amsart}    
\usepackage{latexsym, amsmath, amsthm, amssymb, setspace, verbatim}
\usepackage[all]{xy}
\usepackage[utf8]{inputenc} \usepackage[T1]{fontenc} \usepackage{lmodern}
\usepackage{hyperref}
\usepackage{enumitem}

\title{ Strongly self-absorbing \cstar-dynamical systems }
\author{ Gábor Szabó }

\address{KU Leuven, Department of Mathematics, Celestijnenlaan 200B \phantom{------}\linebreak 
\text{}\hspace{3mm} B-3001 Leuven, Belgium \phantom{------} (Current address)}
\email{gabor.szabo@kuleuven.be}
\thanks{\emph{Supported by:} SFB 878 \emph{Groups, Geometry and Actions}}

\numberwithin{equation}{section}

\begin{document}

\newcommand\set[1]{\left\{#1\right\}}  

\newcommand{\IC}[0]{\mathbb{C}} 
\newcommand{\IN}[0]{\mathbb{N}}
\newcommand{\IQ}[0]{\mathbb{Q}} 
\newcommand{\IR}[0]{\mathbb{R}}
\newcommand{\IT}[0]{\mathbb{T}}
\newcommand{\IZ}[0]{\mathbb{Z}}

\newcommand{\CC}[0]{\mathcal{C}} 
\newcommand{\CD}[0]{\mathcal{D}}
\newcommand{\CM}[0]{\mathcal{M}} 
\newcommand{\CO}[0]{\mathcal{O}}  
\newcommand{\CR}[0]{\mathcal{R}}
\newcommand{\CU}[0]{\mathcal{U}}
\newcommand{\CZ}[0]{\mathcal{Z}}

\newcommand{\fp}[0]{\mathfrak{p}}
\newcommand{\mfi}[0]{\mathfrak{i}}
\newcommand{\mfj}[0]{\mathfrak{j}}

\renewcommand{\phi}[0]{\varphi}
\newcommand{\eps}[0]{\varepsilon}

\newcommand{\quer}[0]{\overline}
\newcommand{\id}[0]{\operatorname{id}}		
\newcommand{\eins}[0]{\mathbf{1}}			
\newcommand{\ad}[0]{\operatorname{Ad}}
\newcommand{\fin}[0]{{\subset\!\!\!\subset}}
\newcommand{\Aut}[0]{\operatorname{Aut}}
\newcommand*\onto{\ensuremath{\joinrel\relbar\joinrel\twoheadrightarrow}} 
\newcommand*\into{\ensuremath{\lhook\joinrel\relbar\joinrel\rightarrow}}  
\newcommand{\dst}[0]{\displaystyle}
\newcommand{\cstar}[0]{$\mathrm{C}^*$}
\newcommand{\dist}[0]{\operatorname{dist}}
\newcommand{\ue}[1]{{~\approx_{\mathrm{u},#1}}~}
\newcommand{\mue}[0]{\approx_{\mathrm{mu}}}
\newcommand{\End}[0]{\operatorname{End}}
\newcommand{\ann}[0]{\operatorname{Ann}}
\newcommand{\strict}[0]{\stackrel{\text{\tiny str}}{\longrightarrow}}
\newcommand{\cc}[0]{\simeq_{\mathrm{cc}}}
\newcommand{\scc}[0]{\simeq_{\mathrm{scc}}}

\newtheorem{satz}{Satz}[section]		
\newtheorem{cor}[satz]{Corollary}
\newtheorem{lemma}[satz]{Lemma}
\newtheorem{prop}[satz]{Proposition}
\newtheorem{theorem}[satz]{Theorem}
\newtheorem*{theoreme}{Theorem}

\theoremstyle{definition}
\newtheorem{conjecture}[satz]{Conjecture}
\newtheorem{defi}[satz]{Definition}
\newtheorem*{defie}{Definition}
\newtheorem{defprop}[satz]{Definition \& Proposition}
\newtheorem{nota}[satz]{Notation}
\newtheorem*{notae}{Notation}
\newtheorem{rem}[satz]{Remark}
\newtheorem*{reme}{Remark}
\newtheorem{example}[satz]{Example}
\newtheorem{defnot}[satz]{Definition \& Notation}
\newtheorem{question}[satz]{Question}
\newtheorem*{questione}{Question}

\newenvironment{bew}{\begin{proof}[Proof]}{\end{proof}}

\begin{abstract} 
We introduce and study strongly self-absorbing actions of locally compact groups on \cstar-algebras. This is an equivariant generalization of a strongly self-absorbing \cstar-algebra to the setting of \cstar-dynamical systems. 
The main result is the following equivariant McDuff-type absorption theorem: A cocycle action $(\alpha,u): G\curvearrowright A$ on a separable \cstar-algebra is cocycle conjugate to its tensorial stabilization with a strongly self-absorbing action $\gamma: G\curvearrowright\CD$, if and only if there exists an equivariant and unital $*$-homomorphism from $\CD$ into the central sequence algebra of $A$. We also discuss some non-trivial examples of strongly self-absorbing actions.
\end{abstract}

\maketitle

\tableofcontents


\setcounter{section}{-1}

\section{Introduction}
\noindent
Group actions on \cstar-algebras and von Neumann algebras are one of the most fundamental
subjects in the theory of operator algebras. Because of this, the problem of classifying group actions has a long history within the theory of operator algebras. A typical example is Connes' classification of injective factors, which involves the classification of cyclic group actions \cite{Connes75, Connes77}. This has sparked a lot of work towards classifying group actions on von Neumann algebras, of which the classification of actions of countable, amenable groups on injective  factors (done by many hands) can be regarded as a highlight, see \cite{Connes75, Connes77, Jones80, Ocneanu85, SutherlandTakesaki89, KawahigashiSutherlandTakesaki92, KatayamaSutherlandTakesaki98}. Although this has historically been developed as a case-by-case study by the factor type, a unified approach for McDuff factors has been found in \cite{Masuda07} building on an intertwining argument \cite{EvansKishimoto97} of Evans-Kishimoto originally invented for the \cstar-setting.

Classification of group actions on \cstar-algebras is, however, still a far less developed subject. This is not least because of substantial $K$-theoretical difficulties, which are already apparent in the seemingly harmless case of finite groups: these can act on a $K$-theoretically trivial algebra such as $\CO_2$ and produce a variety of possible $K$-groups in the crossed product, see \cite{Izumi04, Izumi04II, BarlakSzabo15}. However, despite such difficulties, there is more and more important work going on to make progress on the \cstar-algebraic side. Finite group actions with the Rokhlin property were successfully classified by Izumi \cite{Izumi04, Izumi04II} on classifiable classes of \cstar-algebras.
Phillips is currently developing a theory for the classification of pointwise outer, finite group actions on Kirchberg algebras via equivariant $KK$-theory. In his pioneering work \cite{Kishimoto95, Kishimoto96, Kishimoto98, Kishimoto98II}, Kishimoto has outlined a method to use the Rokhlin property in order to classify single automorphisms (i.e.~$\IZ$-actions) on certain inductive limit \cstar-algebras such as A$\IT$ algebras. This has been developed further by work of Matui \cite{Matui10} for AH algebras, Nakamura \cite{Nakamura00} for Kirchberg algebras, and more recently Lin \cite{Lin15} for TAI algebras. These techniques and results are being pushed to $\IZ^d$-actions and even beyond, see \cite{Nakamura99, KatsuraMatui08, Matui08, Matui10, Matui11, IzumiMatui10, Izumi12OWR}. The reader is recommended to also consult Izumi's survey article \cite{Izumi10} for an informative summary of these developements.

In the non-equivariant classification theory of \cstar-algebras, also known as the Elliott program, a certain class of \cstar-algebras can be conceptually easier to handle under the assumption that they absorb certain \cstar-algebras tensorially. In the general form of the Elliott program, this relates to Jiang-Su stability of \cstar-algebras and the Toms-Winter conjecture \cite{ElliottToms08, Winter10, WinterZacharias10}. A famous result by Kirchberg-Phillips \cite{KirchbergPhillips00} asserts that all Kirchberg algebras absorb the Cuntz algebra $\CO_\infty$ and are absorbed by the Cuntz algebra $\CO_2$ tensorially, which Phillips' approach \cite{Phillips00} to the classification of these algebras made use of. (See also \cite{KirchbergC} for Kirchberg's more direct approach to this classification.) In this way, one can regard the \cstar-algebras $\CO_2$ and $\CO_\infty$ as cornerstones of the classification of Kirchberg algebras. Even in order to classify stably finite \cstar-algebras, the method of \emph{localizing} the classification at a strongly self-absorbing \cstar-algebra, as outlined by Winter in \cite{Winter14Lin}, has become the state-of-the-art. Strongly self-absorbing \cstar-algebras have historically been looked at by example, but for the first time conceptually fleshed out by Toms-Winter in \cite{TomsWinter07}. A very useful tool from \cite{TomsWinter07} in studying $\CD$-absorbing \cstar-algebras for a strongly self-absorbing \cstar-algebra $\CD$ is a McDuff-type result, a variant of which was also proved by Kirchberg in \cite{Kirchberg04}. We note that the $K_1$-injectivity criterion in both \cite{Kirchberg04} and \cite{TomsWinter07} has turned out to be redundant due to the main result of \cite{Winter11}.

\begin{theoreme}
Let $\CD$ be a strongly self-absorbing \cstar-algebra. A separable \cstar-algebra $A$ is $\CD$-stable if and only if its (corrected) central sequence algebra 
\[
F_\infty(A)=(A_\infty\cap A')/\ann(A,A_\infty)
\] 
contains a unital copy of $\CD$.
\end{theoreme}

We note that although Toms-Winter's criterion \cite[2.3]{TomsWinter07} is not directly stated in these terms, it is equivalent to the above criterion.
It is also important to note that a precurser to this result has been known before due to R{\o}rdam \cite[7.2.2]{Rordam}, and his result recovers the above theorem in the unital case. Characterizing $\CD$-stability in this manner has a number of interesting consequences, for instance permanence properties for the class of $\CD$-stable \cstar-algebras, which were studied both in \cite{TomsWinter07} and \cite{Kirchberg04}. The reason for calling this a \emph{McDuff-type} absorption result comes from a famous result of McDuff \cite{McDuff70}, which characterizes when a II${}_{1}$-factor absorbs the hyperfinite II${}_{1}$-factor tensorially:

\begin{theoreme}
Let $M$ be a $\mathrm{II}_1$-factor with separable predual. Let $\CR$ be the hyperfinite $\mathrm{II}_1$-factor. Let $\omega$ be a free ultrafilter on $\IN$. Then the following are equivalent:
\begin{enumerate}[label=\textup{(\roman*)}, leftmargin=*]
\item $M\cong M\bar{\otimes}\CR$.
\item $M^\omega\cap M'$ contains a unital copy of $\CR$.
\item $M^\omega\cap M'$ contains a unital copy of $M_k$, for some $k\in\IN$.
\end{enumerate}
\end{theoreme}

In this paper, we introduce strongly self-absorbing actions of locally compact groups on \cstar-algebras, extending the notion of a strongly self-absorbing \cstar-algebra to the setting of \cstar-dynamical systems. Our main result is the following equivariant McDuff-type absorption result for strongly self-absorbing actions:

\begin{samepage}
\begin{theoreme}[cf.~\ref{equi-McDuff}]
Let $A$ be a separable \cstar-algebra, and let $\CD$ be a separable, unital \cstar-algebra. Let $\alpha: G\curvearrowright A$ be an action and $\gamma: G\curvearrowright\CD$ a strongly self-absorbing action.
Then $(A,\alpha)$ is cocycle conjugate to $(A\otimes\CD,\alpha\otimes\gamma)$ if and only if there exists a unital, equivariant $*$-homomorphism from $\CD$ to $F_\infty(A)$. (Here, $F_\infty(A)$ is endowed with the $G$-action that is induced by component-wise application of $\alpha$ to representing sequences.) 
\end{theoreme}
\end{samepage}

The main result \ref{equi-McDuff} is in fact more general than stated here, because we also consider cocycle actions $(\alpha,u): G\curvearrowright A$ for possibly non-trivial 2-cocycles $u$. In this context, cocycle conjugacy turns out to be equivalent to strong cocycle conjugacy.

One should note that special kinds of equivariant tensorial absorption theorems have already existed in the theory of von Neumann algebras for a long time. For instance, in Ocneanu's classification paper \cite{Ocneanu85}, an important step consisted in showing that any pointwise outer action of an amenable group on the hyperfinite $\mathrm{II}_1$-factor $\CR$ absorbs the trivial $G$-action on $\CR$ tensorially.

Moreover, one should also note that the above equivariant McDuff-type absorption result in the \cstar-setting is at least folklore for cocycle actions of discrete groups. For finite group actions, Izumi has alluded to special cases of such an absorption theorem in \cite{Izumi04}.
As an important tool, it has been used to prove equivariant absorption theorems by Izumi-Matui in \cite{IzumiMatui10}, Goldstein-Izumi in \cite{GoldsteinIzumi11}, and Matui-Sato in \cite{MatuiSato12_2, MatuiSato14}. 

Let us now summarize how this paper is organized. 
In the first section, we remind the reader of some standard definitions from the literature, introduce terminology and recall some sequence algebra techniques.

In the second section, we prove a preliminary version of the main result, which only treats the unital case and is easier to prove. In order to do this, we start the section by proving an equivariant version of a well-known one-sided intertwining argument, which will also be crucial in the later sections.
In the second section, we must also deal with an important technical obstacle: If one wishes to apply a typical sequence algebra argument (commonly refered to as the \emph{reindexation trick}) in the equivariant context, a non-discrete group may cause problems. This stems from the fact that for a non-discrete group $G$ and a continuous action $\alpha: G\curvearrowright A$ on a \cstar-algebra, the induced (algebraic) action $\alpha_\infty: G\curvearrowright A$ on the sequence algebra will typically fail to be point-norm continuous. Moreover, given an element $x\in A_\infty$ satisfying some desirable dynamical property involving the action $\alpha_\infty$, it is at first unclear what kind of \emph{approximate} version of this property holds along a representing sequence $(x_n)_n\in\ell^\infty(\IN,A)$ of $x$. After all, naively lifting relations to the representing sequence will only allow keeping track of some property on finitely many group elements of $G$ at a time, and this is not good enough. This may be even more troublesome for a representing sequence of some element $x\in F_\infty(A)$ in the central sequence algebra of $A$.
However, we solve this problem by borrowing an idea of Guentner-Higson-Trout \cite{GuentnerHigsonTrout} to apply a Baire-category argument on representing sequences, which (simply put) enables us to keep track of a property on compact subsets of the group at a time, instead of only finite sets.

In the third section, we introduce and study strongly self-absorbing actions for a second-countable, locally compact group. We first establish some equivariant generalizations of some partial results from \cite{TomsWinter07}, and then prove the main result (see above), which is an equivariant McDuff-type absorption theorem for strongly self-absorbing actions. 

In the fourth section, we study a property for group actions, which is a slight weakening of being strongly self-absorbing; we will call such actions semi-strongly self-absorbing. As it turns out, actions that are genuinely semi-strongly self-absorbing can only exist in cases where the acting group is non-compact. However, it has some advantages to consider this weaker notion: Firstly, the property of being semi-strongly self-absorbing is more common and easier to verify than the property of being strongly self-absorbing. Secondly, we will see that the above equivariant McDuff-type theorem can be extended to the case where $\gamma$ is only assumed to be semi-strongly self-absorbing. In particular, one may argue that the class of semi-strongly self-absorbing actions deserves just as much attention as the class of strongly self-absorbing actions, at least from the point of view of tensorial absorption.

In the fifth section, we will discuss non-trivial examples of (semi-)strongly self-absorbing actions on \cstar-algebras, and present some open problems regarding certain classes of model actions.

I am confident that the main results and techniques of this paper will become important tools in some kind of equivariant Elliott program for nice group actions on classifiable \cstar-algebras. Given the astounding importance that strongly self-absorbing \cstar-algebras have in the Elliott program, I expect the same to become true for (semi-)strongly self-absorbing actions in the classification theory of group actions.

I am very grateful to Hiroki Matui for showing me a proof of a variant of the important technical Lemma \ref{one side}, and for some helpful discussions.

\begin{reme}
This is an arxiv-exclusive version of the article, uploaded after it was published in Transactions of the American Mathematical Society, 2018, volume 370, pages 99--130.
It fixes a mistake that made it to the published version. 
A claim formerly included as ``Corollary 1.16'' was false as stated: It claimed that under certain assumptions on \cstar-dynamical systems, one can conclude that they are conjugate, but in general one may only conclude that they are (strongly) cocycle conjugate. (The correct version of this statement will be made rigorous in forthcoming work.)
Various proofs in this version are slighty modified as to not need that type of statement in the first place, and hence the main results remain true as they were stated in the original version.
It is my intention to submit these corrections in the form of a corrigendum to the journal. 

This should perhaps be considered as a \emph{bare minimum} revision:
Other than implementing such corrections and updating the original references, no effort was undertaken to make further improvements on style, notation, exposition, or simplifying proofs.
\end{reme}


\section{Preliminaries}

\begin{nota}
Unless specified otherwise, we will stick to the following notational conventions in this paper:
\begin{itemize}
\item $A$ and $B$ denote \cstar-algebras.
\item $\CM(A)$ denotes the multiplier algebra of $A$ and $\tilde{A}$ denotes the smallest unitalization of $A$.
\item Let $m_\lambda\in\CM(A)$ be a net and $m\in\CM(A)$ an element. If $m_\lambda$ converges to $m$ in the strict topology, i.e.~$m_\lambda\cdot a\to m\cdot a$ and $a\cdot m_\lambda\to a\cdot m$ in norm for every $a\in A$, then we write $m_\lambda\strict m$.
\item $G$ denotes a second-countable, locally compact group.
\item The symbol $\alpha$ is used either for a continuous action $\alpha: G\curvearrowright A$ or more generally for a point-norm continuous map $\alpha: G\to\Aut(A)$, e.g.~if $\alpha$ belongs to a cocycle action, as defined below. 
By slight abuse of notation, we will also write $\alpha: G\to\Aut(\CM(A))$ for the unique strictly continuous extension.
\item If $\alpha: G\curvearrowright A$ is an action, then $A^\alpha$ denotes the fixed-point algebra of $A$.
\item If $(X,d)$ is some metric space with elements $a,b\in X$, then we write $a=_\eps b$ as a shortcut for $d(a,b)\leq\eps$.
\end{itemize}
\end{nota}

\begin{defi}[cf.~{\cite[2.1]{BusbySmith70} or \cite[2.1]{PackerRaeburn89}}] 
\label{cocycle actions} 
A cocycle action $(\alpha,u): G\curvearrowright A$ is a point-norm continuous map $\alpha: G\to\Aut(A),~ g\mapsto\alpha_g$ together with a strictly continuous map $u: G\times G\to\CU(\CM(A))$ satisfying the properties
\begin{itemize}
\item $\alpha_1=\id_A$
\item $\alpha_r\circ\alpha_s = \ad(u(r,s))\circ\alpha_{rs}$
\item $u(1,t)=u(t,1)=\eins$
\item $\alpha_r(u(s,t))\cdot u(r,st) = u(r,s)\cdot u(rs,t)$
\end{itemize}
for all $r,s,t\in G$.
\end{defi}

\begin{reme}
If $u=\eins$ is trivial, then this just recovers the definition of an ordinary action $\alpha: G\curvearrowright A$. 
\end{reme}

\begin{defi}[cf.~{\cite[3.2]{PackerRaeburn89}} for (i)+(ii)]
Let $\alpha: G\curvearrowright A$ be an action. Consider a strictly continuous map $w: G\to\CU(\CM(A))$.
\begin{enumerate}[label=(\roman*),leftmargin=*] 
\item $w$ is called an $\alpha$-1-cocycle, if one has $w_g\alpha_g(w_h)=w_{gh}$ for all $g,h\in G$.
In this case, the map $\alpha^w: G\to\Aut(A)$ given by $\alpha_g^w=\ad(w_g)\circ\alpha_g$ is again an action, and is called a cocycle perturbation of $\alpha$. Two $G$-actions on $A$ are called exterior equivalent if one of them is a cocycle perturbation of the other.
\item Assume that $w$ is an $\alpha$-1-cocycle. It is called a coboundary, if there exists a unitary $v\in\CU(\CM(A))$ with $w_g = v\alpha_g(v^*)$ for all $g\in G$.
\item Assume that $w$ is an $\alpha$-1-cocycle. It is called an approximate coboundary, if there exists a sequence of unitaries $x_n\in\CU(\CM(A))$ such that $x_n\alpha_g(x_n^*) \strict w_g$ for all $g\in G$ and uniformly on compact sets. Two $G$-actions on $A$ are called strongly exterior equivalent, if one of them is a cocycle perturbation of the other via an approximate coboundary.
\end{enumerate}
\end{defi}

\begin{defi}[cf.~{\cite[3.1]{PackerRaeburn89}} for (i)]
Let $(\alpha,u), (\beta,w): G\curvearrowright A$ be two cocycle actions. 
\begin{enumerate}[label=(\roman*),leftmargin=*] 
\item The pairs $(\alpha,u)$ and $(\beta,w)$ are called exterior equivalent, if there is a strictly continuous map $v: G\to\CU(\CM(A))$ satisfying $\beta_g=\ad(v_g)\circ\alpha_g$ and $w(s,t)=v_s\alpha_s(v_t)u(s,t)v_{st}^*$ for all $g,s,t\in G$.
\item The pairs $(\alpha,u)$ and $(\beta,w)$ are called strongly exterior equivalent, if there is a map $v: G\to\CU(\CM(A))$ as above such that there is a sequence of unitaries $x_n\in\CU(\CM(A))$ with $x_n\alpha_g(x_n^*) \strict v_g$ for all $g\in G$ and uniformly on compact sets. 
\end{enumerate}
\end{defi}

\begin{defi}
Let $(\alpha,u): G\curvearrowright A$ and $(\beta,w): G\curvearrowright B$ be two cocycle actions. A non-degenerate $*$-homomorphism $\phi: A\to B$ is called equivariant, and is written $\phi: (A,\alpha,u)\to (B,\beta,w)$, if one has $\beta_g\circ\phi=\phi\circ\alpha_g$ and $w(s,t)=\phi(u(s,t))$ for all $g,s,t\in G$.

If $u(\cdot,\cdot)=w(\cdot,\cdot)=\eins$, i.e.~one has actions $\alpha: G\curvearrowright A$ and $\beta: G\curvearrowright B$, then we omit the cocycles in the notation and write $\phi: (A,\alpha)\to (B,\beta)$.
\end{defi}

Let us now recall the notions of conjugacy, cocycle conjugacy and strong cocycle conjugacy. The latter notion is somewhat less prominent in the literature, and was first explicitely introduced by Izumi-Matui in \cite[2.1]{IzumiMatui10} as a relation between discrete group actions. In \cite[2.1(3)]{MatuiSato14}, Matui-Sato extended this notion to be a relation between cocycle actions of discrete groups. It is important to note that implicitely, the notion of strong cocycle conjugacy has played a role in earlier work of Kishimoto, see \cite{Kishimoto95, Kishimoto96, EvansKishimoto97, Kishimoto98, Kishimoto98II}.

We are going to consider a suitable generalization of strong cocycle conjugacy to cocycle actions of locally compact groups. We note, however, that the definition given below only extends the definition from \cite{IzumiMatui10, MatuiSato14} in the case of unital \cstar-algebras, and is weaker in the non-unital case.

\begin{defi} 
Two cocycle actions $(\alpha,u): G\curvearrowright A$ and $(\beta,w): G\curvearrowright B$ are called 
\begin{enumerate}[label=(\roman*),leftmargin=*]
\item conjugate, if there is an equivariant isomorphism $\phi: (A,\alpha,u)\to (B,\beta,w)$.  In this case, we write $(A,\alpha,u)\cong (B,\beta,w)$.
\item cocycle conjugate, if there is an isomorphism $\phi: A\to B$ such that $(\phi\circ\alpha\circ\phi^{-1}, \phi\circ u)$ is exterior equivalent to $(\beta,w)$. In this case, we write $(A,\alpha,u)\cc (B,\beta,w)$.
\item strongly cocycle conjugate, if there is an isomorphism $\phi: A\to B$ such that $(\phi\circ\alpha\circ\phi^{-1}, \phi\circ u)$ is strongly exterior equivalent to $(\beta,w)$. In this case, we write $(A,\alpha,u)\scc (B,\beta,w)$. 
\end{enumerate}
If $\alpha$ and $\beta$ are genuine actions, we omit the cocycles from the notation and write $(A,\alpha)\cong (B,\beta)$, $(A,\alpha)\cc (B,\beta)$, or $(A,\alpha)\scc (B,\beta)$.
\end{defi}

As we have mentioned before, the above given definition of strong cocycle conjugacy does not extend the definition from \cite{IzumiMatui10, MatuiSato14} in the case of non-unital \cstar-algebras. Our definition is weaker, and the difference is that the approximate coboundary condition is expressed in the strict topology rather than the norm topology.

As was outlined in the introduction, we aim to characterize equivariant absorption of certain actions via a property of the central sequence algebra. In particular, we shall spend most of our effort on studying sequence algebras, (relative) central sequence algebras and actions induced on them. Let us begin by recalling the definition of a \emph{corrected} relative central sequence algebra, i.e.~a quotient of a relative central sequence algebra by an Annihilator ideal. This idea originates from Kirchberg's important work \cite{Kirchberg04} on central sequences of \cstar-algebras.

\begin{defi}[cf.~{\cite[1.1]{Kirchberg04}}] 
Let $A$ be a \cstar-algebra. 
Consider the sequence algebra of $A$ as the quotient
\[
A_\infty = \ell^\infty(\IN,A)/\set{ (x_n)_n ~|~ \lim_{n\to\infty}\| x_n\|=0}.
\]
There is a standard embedding of $A$ into $A_\infty$ by sending an element to its constant sequence.
For some \cstar-subalgebra $B\subset A_\infty$, the relative central sequence algebra is defined as
\[
A_\infty\cap B' = \set{x\in A_\infty ~|~ xb=bx~\text{for all}~b\in B}.
\]
Observe that the two-sided annihilator
\[
\ann(B,A_\infty) = \set{x\in A_\infty ~|~ xb=bx=0~\text{for all}~b\in B}
\]
is an ideal in $A_\infty\cap B'$, and one can thus define
\[
F(B,A_\infty) = A_\infty\cap B'/\ann(B,A_\infty).
\]
Of particular importance is the central sequence algebra of $A$, i.e.~$A_\infty\cap A'$, or $F_\infty(A)=F(A,A_\infty)$, or the relative central sequence algebra $\CM(A)_\infty\cap A'$.
\end{defi}

\begin{rem} 
\label{induced-maps-and-unitaries}
If $\phi\in\Aut(A)$ is an automorphism, then component-wise application on representative sequences yields a well-defined automorphism $\phi_\infty\in\Aut(A_\infty)$ or $\phi_\infty\in\Aut(\CM(A)_\infty)$. If $B\subset A_\infty$ is $\phi_\infty$-invariant, then $\phi_\infty$ restricts to a well-defined automorphism on $A_\infty\cap B'$ and $\ann(B,A_\infty)$. This further induces an automorphism $\tilde{\phi}_\infty$ on $F(B,A_\infty)$.

In the special case that $\phi=\ad(u)$ is an inner automorphism for some $u\in\CM(A)$ and we have $uB+Bu\subset B$, the induced automorphism $\tilde{\phi}_\infty$ on $F(B, A_\infty)$ is the identity map.
\end{rem}
\begin{proof}
We have $u\cdot b, b\cdot u\in B$ for all $b\in B$ by assumption. Given $x\in A_\infty\cap B'$ and $b\in B$, we thus have
\[
(uxu^*)b = ux(u^*b) = u(u^*b)x = bx = xb.
\]
But this implies that $x-uxu^*\in\ann(B,A_\infty)$, which shows that $\ad(u)$ indeed acts trivially on $F(B,A_\infty)$.
\end{proof}

\begin{nota}
Let $(\alpha,u): G\curvearrowright A$ be a cocycle action. 
In view of \ref{induced-maps-and-unitaries}, we may consider the induced map $\alpha_\infty: G\to\Aut(A_\infty),~g\mapsto \alpha_{\infty,g}:=(\alpha_g)_\infty$. In general, this map is not point-norm continuous. So let
\[
A_{\infty,\alpha} = \set{ x\in A_\infty ~|~ [g\mapsto\alpha_{\infty,g}]~\text{is continuous} }
\]
be the \cstar-subalgebra of elements in $A_\infty$ on which this map acts continuously.
\end{nota}

\begin{rem}
\label{cont-central-seq}
Let $G$ be a second-countable, locally compact group. Let $(\alpha,u): G\curvearrowright A$ be a cocycle action. Assume that $B\subset A_\infty$ is a \cstar-subalgebra that is invariant under multiplication with $\set{u(s,t)}_{s,t\in G}$ and under applying $\set{\alpha_{\infty,s}}_{s\in G}$. Then it follows from \ref{induced-maps-and-unitaries} that the induced family of automorphisms $\set{\tilde{\alpha}_{\infty,g}}_{g\in G}$ defines a genuine $G$-action on $F(B,A_\infty)$. However, this action will in general fail to be point-norm continuous. We can thus consider the continuous part
\[
F_\alpha(B,A_\infty) = \set{ x\in F(B,A_\infty) ~|~ [G\ni g\mapsto \tilde{\alpha}_{\infty,g}(x)] ~\text{is continuous} }.
\]
Now let $B$ be another \cstar-algebra and $(\beta,v): G\curvearrowright B$ another cocycle action. Assume that $(A,\alpha,u)$ and $(B,\beta,v)$ are cocycle conjugate via an isomorphism $\psi: A\to B$. Then for every $g\in G$, the automorphisms $\beta_g$ and $\phi\circ\alpha_g\circ\phi^{-1}$ agree up to an inner automorphism, so their induced maps on $F_\infty(B)$ are the same. In particular, component-wise application of $\phi$ on representing sequences yields an equivariant isomorphism
\[
\big( F_\infty(A), \tilde{\alpha}_\infty \big) ~\cong~ \big( F_\infty(B), \tilde{\beta}_\infty \big).
\]
\end{rem}

\begin{rem}
\label{F(A)-enhanced}
If $B\subset A_\infty$ is a \cstar-subalgebra, then obviously $\ann(B,A_\infty)=\ann(B,\CM(A)_\infty)\cap A_\infty$ and $A_\infty\cap B'= (\CM(A)_\infty\cap B')\cap A_\infty$. So we have a natural embedding $F(B,A_\infty)\to F(B,\CM(A)_\infty)$. If $B$ is $\sigma$-unital, then this embedding is an isomorphism.

The same is true if we replace $\CM(A)$ by $\tilde{A}$.
\end{rem}
\begin{proof}
Let $\tilde{x}=x+\ann(B,\CM(A)_\infty)\in F(B,\CM(A)_\infty)$ for some $x\in\CM(A)_\infty\cap B'$ be represented by a bounded sequence $x_n\in\CM(A)$. Since $B$ is $\sigma$-unital, let $h\in B$ be a strictly positive element. Let $h$ be represented by a bounded sequence $(h_n)_n$ in $A$. Using that $A$ is strictly dense in $\CM(A)$, we choose elements $y_n\in A$ with $\|y_n\|\leq\|x_n\|$ satisfying
\[
x_n h_n =_{1/n} y_n h_n~\text{and}~ h_n x_n =_{1/n} h_n y_n\quad\text{for all}~n\in\IN.
\]
For the resulting bounded sequence and its induced element $y=[(y_n)]\in A_\infty$, we thus have $yh=xh$ and $hy=hx$. But since $h$ is strictly positive in $B$, this implies $yb=xb$ and $by=bx$ for all $b\in B$. In particular,
\[
y\in A_\infty\cap B'~\text{and}~bx=by\quad\text{for all}~b\in B.
\]
But this means that $\tilde{x}=y+\ann(B,\CM(A)_\infty)\in F(B,\CM(A)_\infty)$ is in the image of $F(B,A_\infty)$ under the natural embedding. This shows our claim.
\end{proof}

\begin{rem}[cf.~{\cite[1.9(1)+(4)]{Kirchberg04}}]
\label{kirchberg-projection}
For a \cstar-algebra $A$, we denote $D_{\infty,A}=\quer{A\cdot A_\infty\cdot A}$, the hereditary \cstar-subalgebra of $A_\infty$ generated by $A$. Let $B\subset A$ be a non-degenerate \cstar-subalgebra. Then one checks easily that for any $y\in A_\infty\cap B'$, we have $yD_{\infty,A}+D_{\infty,A}y\subset D_{\infty,A}$.
By the universal property of the multiplier algebra, we have a natural $*$-homomorphism
\[
\pi: A_\infty\cap B' \to \CM( D_{\infty,A}),\quad y\mapsto [x\mapsto yx].
\]
The kernel coincides with $\ann(B,A_\infty)$ by definition, and thus we get a natural embedding
\[
F(B,A_\infty) = (A_\infty\cap B')/\ann(B,A_\infty) \to \CM(D_{\infty,A}).
\]
In particular, we can view $F(B,A_\infty)$ as a subset of $\CM(D_{\infty,A})$ in a natural way. Given $d\in D_{\infty,A}$ and $x\in F(B,A_\infty)$, we can multiply $d\cdot x=d\cdot y$ for some element $y$ with $\pi(y)=x$. Analogously, the product $x\cdot d$ is well-defined, and both give well-defined values in $D_{\infty,A}\subset A_\infty$.
\end{rem}

\begin{defi}
Let $\alpha: G\curvearrowright A$ and $\beta: G\curvearrowright B$ be two actions of a second-countable, locally compact group on separable \cstar-algebras. Let $\phi_1,\phi_2: (A,\alpha)\to (B,\beta)$ be two equivariant $*$-homomorphisms. The maps $\phi_1,\phi_2$ are called approximately $G$-unitarily equivalent, if there is a sequence $v_n\in\CU(\CM(B))$ with $\phi_2 = \lim_{n\to\infty} \ad(v_n)\circ\phi_1$ and $\beta_g(v_n)-v_n \strict 0$ uniformly on compact subsets of $G$. We denote $\phi_1 \ue{G} \phi_2$.
\end{defi}

\section{A one-sided intertwining argument}

\noindent 
The following is an equivariant analogue of a well-known Lemma in R{\o}rdam's book \cite{Rordam}.
A variant of the next Lemma has been shown to me by Matui, and I am grateful to him for kindly allowing me include it in this paper.

\begin{lemma}[cf.~{\cite[2.3.5]{Rordam}}] 
\label{one side}
Let $G$ be a second-countable, locally compact group. Let $(\alpha,u): G\curvearrowright A$ and $(\beta,w): G\curvearrowright B$ be two cocycle actions on separable \cstar-algebras. Let $\phi: (A,\alpha,u)\to (B,\beta,w)$ be an injective, non-degenerate and equivariant $*$-homomorphism. Assume the following:

For every $\eps>0$, compact subset $K\subset G$ and finite subsets $F_A\fin A, F_B\fin B$, there exists a unitary $z\in\CU(\CM(B))$ satisfying
\begin{enumerate}[label=\textup{(\ref*{one side}\alph*)}]
\item $\|[z, \phi(a)]\|\leq\eps$ for every $a\in F_A$. \label{eq:oneside-a}
\item $\dist(z^*bz, \phi(A)) \leq\eps$ for every $b\in F_B$. \label{eq:oneside-b}
\item $b\beta_g(z) =_\eps bz$ for every $g\in K$ and $b\in F_B$. \label{eq:oneside-c}
\end{enumerate}
Then $\phi$ is approximately multiplier unitarily equivalent to an isomorphism $\psi: A\to B$ inducing strong cocycle conjugacy between $(A,\alpha,u)$ and $(B,\beta,w)$.

If moreover $z$ may always be chosen to be a unitary in $\tilde{B}$, then $\phi$ is approximately unitarily equivalent to an isomorphism inducing strong cocycle conjugacy between $(A,\alpha,u)$ and $(B,\beta,w)$.
\end{lemma}

\begin{proof}
Let $\set{a_n}_{n\in\IN}\subset A$ and $\set{b_n}_{n\in\IN}\subset B$ be dense sequences. Since $G$ is $\sigma$-compact, write $G=\bigcup_{n\in\IN} K_n$ for an increasing union of compact subsets $1_G\in K_n$. We are going to add unitaries to $\phi$ step by step:

In the first step, choose some $a_{1,1}\in A$ and $z_1\in\CU(\CM(B))$ such that
\begin{itemize}
\item $z_1^*b_1z_1 =_{1/2} \phi(a_{1,1})$
\item $\|[z_1,\phi(a_1)]\|\leq 1/2$.
\item $b_1\beta_g(z_1)=_{1/2} b_1z_1$ for all $g\in K_1$.
\end{itemize}
In the second step, choose $a_{2,1},a_{2,2}\in A$ and $z_2\in\CU(\CM(B))$ such that
\begin{itemize}
\item $z_2^*(z_1^*b_jz_1)z_2 =_{1/4} \phi(a_{2,j})$ for $j=1,2$
\item $\|[z_2,\phi(a_j)]\|\leq 1/4$ for $j=1,2$
\item $\|[z_2,\phi(a_{1,1})]\|\leq 1/4$
\item $(b_jz_1)\beta_{g}(z_2) =_{1/4} (b_jz_1)z_2$ for all $g\in K_2$ and $j=1,2.$
\end{itemize}
Now assume that for some $n\in\IN$, we have found $z_1,\dots,z_n\in\CU(\CM(B))$ and $\set{a_{m,j}}_{j\leq m\leq n}\subset A$ satisfying
\begin{enumerate}[label=({c}\arabic*)]
\item $z_n^*(z_{n-1}^*\cdots z_1^*b_jz_1\cdots z_{n-1})z_n =_{2^{-n}} \phi(a_{n,j})$ for $j\leq n$ \label{eq1:c1}
\item $\|[z_n,\phi(a_j)]\|\leq 2^{-n}$ for $j\leq n$ \label{eq1:c2}
\item $\|[z_n,\phi(a_{m,j})]\|\leq 2^{-n}$ for $m<n$ and $j\leq m$ \label{eq1:c3}
\item $(b_jz_1\cdots z_{n-1})\beta_{g}(z_n) =_{2^{-n}} (b_jz_1\cdots z_{n-1})z_n$ for all $g\in K_n$ and $j\leq n$. \label{eq1:c4}
\end{enumerate}
Then we can again apply our assumptions to find $z_{n+1}\in\CU(\CM(B))$ and $\set{a_{n+1,j}}_{j\leq n+1}$ with
\begin{itemize}
\item $z_{n+1}^*(z_{n}^*\cdots z_1^*b_jz_1\cdots z_{n})z_{n+1} =_{2^{-(n+1)}} \phi(a_{n+1,j})$ for $j\leq n+1$
\item $\|[z_{n+1},\phi(a_j)]\|\leq 2^{-(n+1)}$ for $j\leq n+1$
\item $\|[z_{n+1},\phi(a_{m,j})]\|\leq 2^{-(n+1)}$ for $m<n+1$ and $j\leq m$.
\item $(b_j z_1\cdots z_{n})\beta_{g}(z_{n+1}) =_{2^{-(n+1)}} (b_j z_1\cdots z_{n})z_{n+1}$ for all $g\in K_{n+1}$ and $j\leq n+1$.
\end{itemize}
Carry on inductively. Then we define $\psi_n: A\to B$ by the formula $\psi_n = \ad(z_1\cdots z_n)\circ\phi$. Now let us observe a number of facts:

By condition \ref{eq1:c2}, the sequence $(\psi_n(a_j))_{n\in\IN}$ is Cauchy for all $j\in\IN$. Since the $a_j$ are dense, this implies that the $\psi_n$ converge to some $*$-homomorphism $\psi: A\to B$ in the point-norm topology. Then $\psi$ is clearly approximately multplier unitarily equivalent to $\phi$.

By condition \ref{eq1:c3}, we have for every $n\in\IN$ and $j\leq n$ the estimate
\[
\|\psi(a_{n,j})-\psi_{n}(a_{n,j})\| \leq \sum_{k=n+1}^\infty 2^{-k} = 2^{-n}.
\]
Combined with condition \ref{eq1:c1}, we have for every $n\in\IN$ and $j\leq n$ that
\[
\|b_j-\psi(a_{n,j})\| \leq 2^{-n}+\|b_j-\psi_{n}(a_{n,j})\| \leq 2^{-n}+2^{-n} = 2^{1-n}.
\]
Since the $\set{b_j}_j\subset B$ are dense, it follows that $\psi(A)$ is dense, hence $\psi$ is surjective. Since it is also injective, it is an isomorphism.

By condition \ref{eq1:c4}, we have that the sequences $b_j\cdot z_1\cdots z_n \beta_g\big( z_n^*\cdots z_1^* \big)$ are Cauchy sequences for every $j\in\IN$ and $g\in G$, with uniformity on compact sets. Since the $b_j$ form a dense subset of $B$ by choice, it follows that every sequence of functions of the form $[g\mapsto b\cdot z_1\cdots z_n\beta_g(z_n^*\cdots z_1^*)]$ (for $b\in B$) converges uniformly on compact sets. Since $\beta$ is point-norm continuous, it follows that the functions 
\[
g\mapsto \beta_g\Big( \beta_{g^{-1}}(b)^*\cdot z_1\cdots z_n\beta_{g^{-1}}(z_n^*\cdots z_1^*) \Big)^* = z_1\cdots z_n\beta_g(z_n^*\cdots z_1^*)\cdot b
\]
must also converge uniformly on compact sets of $G$, for every $b\in B$.

Denote $x_n = z_1\cdots z_n$. It follows that the strict limit $v_g = \lim_{n\to\infty} x_n\beta_g(x_n^*)$ exists for every $g\in G$, and that this convergence is uniform on compact subsets of $G$. In particular, the assignment $g\mapsto v_g\in\CU(\CM(B))$ is strictly continuous. 

We calculate for each $g\in G$ that
\[
\begin{array}{ccl}
\psi\circ\alpha_g &=& \dst\lim_{n\to\infty} \ad(x_n)\circ\phi\circ\alpha_g \\\\
&=& \dst\lim_{n\to\infty} \ad(x_n) \circ \beta_g \circ \phi \\\\
&=& \dst\lim_{n\to\infty} \ad(x_n)\circ\beta_g\circ\ad(x_n^*)\circ\ad(x_n)\circ\phi \\\\
&=& \dst\lim_{n\to\infty} \ad(x_n\beta_g(x_n^*))\circ\beta_g\circ\ad(x_n)\circ\phi \\\\
&=& \ad(v_g)\circ\beta_g\circ\psi.
\end{array}
\]
Keeping in mind that $\phi,\psi_n: A\to B$ are non-degenerate, we may consider the unique, strictly continuous extentions $\phi, \psi_n: \CM(A)\to\CM(B)$. Observe that we have convergence $\psi_n(m)\strict\psi(m)$ for every $m\in\CM(A)$: Given some $b\in B$ and $\delta>0$, if we choose $n$ large enough such that $\psi_n(m\cdot\psi^{-1}(b))=_\delta \psi(m\cdot\psi^{-1}(b))=\psi(m)b$ and $b=_\delta\psi_n(\psi^{-1}(b))$, then we get
\[
\psi_n(m)b =_\delta \psi_n(m)\cdot\psi_n(\psi^{-1}(b)) = \psi_n(m\cdot\psi^{-1}(b)) =_\delta \psi(m)\cdot b.
\] 
In particular, we see that 
\[
x_n w(s,t) x_n^* = x_n\phi(u(s,t))x_n^* = \psi_n(u(s,t)) \strict \psi(u(s,t))
\]
for every $s,t\in G$.
 
Using this, let us now observe for every $s,t\in G$ that
\[
\begin{array}{ccl}
v_s\beta_s(v_t)w(s,t)v_{st}^* &=& \dst \mathrm{str}\!-\!\!\!\lim_{n\to\infty} x_n\beta_s(x_n^*)\beta_s(x_n\beta_t(x_n^*))w(s,t)(x_n\beta_{st}(x_n^*))^* \\\\
&=& \dst \mathrm{str}\!-\!\!\!\lim_{n\to\infty} x_n(\beta_s\circ\beta_t)(x_n^*) \underbrace{w(s,t)\beta_{st}(x_n)}_{=(\beta_s\circ\beta_t)(x_n)w(s,t)} x_n^* \\\\
&=& \dst \mathrm{str}\!-\!\!\!\lim_{n\to\infty} x_n w(s,t) x_n^* ~=~ \psi(u(s,t)).
\end{array}
\]
This finishes the proof of the first statement. However, reviewing the proof, it is clear that if $z_n$ may always be chosen to be in $\tilde{B}$, then $\phi$ is approximately unitarily equivalent to $\psi$. So the second statement follows as well.
\end{proof}

For the next few results, recall the conventions introduced in Remark \ref{induced-maps-and-unitaries}:

\begin{lemma}
\label{sequence-lift}
Let $G$ be a second-countable, locally compact group. Let $(\alpha,u): G\curvearrowright A$ and $(\beta,w): G\curvearrowright B$ be two cocycle actions. Let $\phi: (A,\alpha,u)\to (B,\beta,w)$ be a non-degenerate and equivariant $*$-homomorphism.

Let $x\in\CM(B)_{\infty}\cap\phi(A)'$ with the property that the map $[g\mapsto \phi(a)\cdot\beta_{\infty,g}(x)]$ is continuous for all $a\in A$. 
Let $(x_n)_n\in\ell^\infty(\IN,\CM(B))$ be a bounded sequence representing $x$.
Then for every $g_0\in G$, $b\in B$ and $\delta>0$, there exists $n_0\in\IN$ and an open neighbourhood $U$ of $g_0$ such that
\[
\sup_{k\geq n_0}~ \sup_{g\in U} \|b(\beta_g(x_k)-\beta_{g_0}(x_k)) \| \leq \delta.
\]
\end{lemma}
\begin{proof}
Since $\phi(A)\subset B$ is assumed to be a non-degenerate subalgebra, it suffices to show the assertion for all $b\in\phi(A)$.

We are going to apply a Baire category argument in the spirit of \cite[1.8]{GuentnerHigsonTrout}. However, we first make a few basic observations concerning cocycle action formulas in general. For $g\in G$, the cocycle identity of $u$ implies
\[
\ad(w(g^{-1},g)) = \ad(w(g^{-1},g))\circ\beta_1 = \beta_{g^{-1}}\circ\beta_{g}
\]
In particular, we have $\beta_{g}^{-1} = \ad(w(g^{-1},g)^*)\circ\beta_{g^{-1}}$ for every $g\in G$. If now $g_0\in G$ is a possibly different group element, then this implies
\[
\beta_g^{-1}\circ\beta_{g_0} = \ad(w(g^{-1},g)^*)\circ\beta_{g^{-1}}\circ\beta_{g_0} = \ad(w(g^{-1},g)^*w(g^{-1},g_0))\circ\beta_{g^{-1}g_0}.
\]

Now let us prove the assertion. Let $g_0\in G, a\in A$ and $\delta>0$ be given. Without loss of generality, assume that all of the elements $x$, $x_k$ and $a$ are contractions. Recall that $\beta: G\to\Aut(B)$ is point-norm continuous. Hence let us find a compact neighbourhood $V_1=V_1^{-1}\ni 1_G$ of the unit such that
\begin{equation} \label{eq1:small-nh-1}
\beta_{g}(\phi(\alpha_{g_0}^{-1}(a)))=_\delta \phi(\alpha_{g_0}^{-1}(a))\end{equation} 
for all $g\in V_1$.

Recall that $u: G\times G\to\CM(A)$ is strictly continuous. So as $V_1$ is compact, the set
\[
K_A= \set{ (\alpha_{g_0}\circ\alpha_{g_1}\circ\alpha_{g_2})^{-1} \Big( a\cdot u(g_0,g_1g_2)^* \alpha_{g_0}\big(u(g_1,g_2)^*\big) \Big) ~|~ g_1,g_2\in V_1  }
\]
is compact in $A$. By equivariance of $\phi$, this means that the set
\[
K_B = \set{ (\beta_{g_0}\circ\beta_{g_1}\circ\beta_{g_2})^{-1} \Big( \phi(a)\cdot w(g_0,g_1g_2)^* \beta_{g_0}\big( w(g_1,g_2)^* \big) \Big) ~|~ g_1, g_2\in V_1  }
\]
is equal to $\phi(K_A)$ and thus compact in $\phi(A)$. Since $x\in\CM(B)_\infty\cap\phi(A)'$, we can find some $n_1\in\IN$ such that
\[
\|[x_k,d]\|\leq\delta\quad\text{for all}~k\geq n_1~\text{and}~d\in K_B.
\]
So if $d\in K_B$ is an element of the above form, one may apply $(\beta_{g_0}\circ\beta_{g_1}\circ\beta_{g_2})$ to this commutator and obtain that
\begin{equation} \label{eq1:small-commutator}
\Big\| \Big[ \phi(a)\cdot w(g_0,g_1g_2)^* \beta_{g_0}\big( w(g_1,g_2)^* \big) , (\beta_{g_0}\circ\beta_{g_1}\circ\beta_{g_2})(x_k) \Big] \Big\| \leq\delta
\end{equation}
for all $g_1,g_2\in V_1$ and $k\geq n_1$.

By the continuity assumption on $x$, we can find a neighbourhood $V=V^{-1}$ of $1_G\in G$ such that
\begin{equation} \label{eq1:cont-assumption}
\|\phi(\alpha_{g_0}^{-1}(a))(\beta_{\infty,g}(x)-x) \| = \limsup_{n\to\infty}~\|\phi(\alpha_{g_0}^{-1}(a))(\beta_g(x_n)-x_n)\|\leq\delta
\end{equation}
for all $g\in V$. In other words,
\[
\inf_{\ell\in\IN}~\sup_{n\geq \ell}~\|\phi(\alpha_{g_0}^{-1}(a))(\beta_g(x_n)-x_n)\|\leq\delta
\]
for all $g\in V$.
For every $\ell\in\IN$, we define
\begin{equation} \label{eq1:W-sets}
W_\ell = \set{ g\in G ~\big|~ \| \phi(\alpha_{g_0}^{-1}(a))(\beta_g(x_k)-x_k)\|\leq 2\delta ~\text{for all}~k\geq \ell}.
\end{equation}
As the extension of $\beta$ to a map from $G$ to $\Aut(\CM(B))$ is strictly continuous, these sets are closed subsets of $G$. By \eqref{eq1:cont-assumption}, it follows that the countable union $\bigcup_{\ell\in\IN} W_\ell\cap W_\ell^{-1}$ contains the neighbourhood $V\cap V_1$ of the unit. 
By local compactness of $G$, it follows from the Baire category theorem that there exists $\ell\in\IN$ such that $W_\ell\cap W_\ell^{-1}$ contains an open subset of $V_1$. Then $U_0=(V_1\cap W_\ell\cap W_\ell^{-1})^{-1}\cdot(V_1\cap W_\ell\cap W_\ell^{-1})$ is a neighbourhood of the unit.  We then have for every $g_1,g_2\in V_1\cap W_\ell\cap W_\ell^{-1}$ and all $k\geq \ell$
\[
\begin{array}{ccl}
\multicolumn{3}{l}{ \phi(a) \cdot \beta_{g_0g_1^{-1}g_2}(x_k) } \\\\
&=& \phi(a) \left( \ad(w(g_0,g_1^{-1}g_2)^*)\circ\beta_{g_0}\circ\beta_{g_1^{-1}g_2}(x_k) \right)  \\\\
&=& \phi(a) \left( \ad\big( w(g_0,g_1^{-1}g_2)^*\cdot\beta_{g_0}(w(g_1^{-1},g_2)^*) \big)\circ\beta_{g_0}\circ\beta_{g_1^{-1}}\circ\beta_{g_2}(x_k) \right)  \\\\
&\stackrel{\eqref{eq1:small-commutator}}{=_\delta}& (\beta_{g_0}\circ\beta_{g_1^{-1}}\circ\beta_{g_2})(x_k)\cdot\phi(a)   \\\\
&\stackrel{\eqref{eq1:small-commutator}}{=_\delta}& \phi(a)\cdot (\beta_{g_0}\circ\beta_{g_1^{-1}}\circ\beta_{g_2})(x_k)   \\\\
&=& \beta_{g_0}\left( \phi(\alpha_{g_0}^{-1}(a))\cdot(\beta_{g_1^{-1}}\circ\beta_{g_2})(x_k) \right) \\\\
&\stackrel{\eqref{eq1:small-nh-1}}{=_\delta}& (\beta_{g_0}\circ\beta_{g_1^{-1}}) \left( \phi(\alpha_{g_0}^{-1}(a)) \cdot \beta_{g_2}(x_k) \right) \\\\
&\stackrel{\eqref{eq1:W-sets}}{=_{2\delta}}& (\beta_{g_0}\circ\beta_{g_1^{-1}}) \left( \phi(\alpha_{g_0}^{-1}(a)) \cdot x_k \right) \\\\
&\stackrel{\eqref{eq1:small-nh-1}}{=_\delta}& \beta_{g_0} \left( \phi(\alpha_{g_0}^{-1}(a)) \cdot \beta_{g_1^{-1}}(x_k) \right) \\\\
&\stackrel{\eqref{eq1:W-sets}}{=_{2\delta}}& \beta_{g_0} \left( \phi(\alpha_{g_0}^{-1}(a)) \cdot x_k \right) \\\\
&=& \phi(a)\cdot\beta_{g_0}(x_k).
\end{array}
\]
In particular, if we set $U=g_0\cdot U_0$, then this shows
\[
 \sup_{k\geq \ell}~ \sup_{g\in U}~ \|\phi(a)\cdot\big( \beta_{g_0}(x_k)-\beta_g(x_k) \big)\| \leq 8\delta.
\]
As $g_0\in G$ and $\delta>0$ were arbitrary, this finishes the proof.
\end{proof}

\begin{lemma}
\label{approx-fixed}
Let $G$ be a second-countable, locally compact group. Let $(\alpha,u): G\curvearrowright A$ and $(\beta,w): G\curvearrowright B$ be two cocycle actions. Let $\phi: (A,\alpha,u)\to (B,\beta,w)$ be a non-degenerate and equivariant $*$-homomorphism. 
Let $z\in\CM(B)_\infty\cap\phi(A)'$ be an element such that $[g\mapsto \phi(a)\beta_{\infty,g}(z)]$ is continuous for every $a\in A$.
Let $\eps>0$, $K\subset G$ a compact set and assume that
$\|\phi(a)(z-\beta_{\infty,g}(z)) \|\leq\eps$ for all $g\in K$ and $a\in A$ with $\|a\|\leq 1$. If $(z_n)_n$ is a bounded sequence representing $z$, then
\[
\limsup_{n\to\infty} ~\max_{g\in K}~ \| b (z_n-\beta_g(z_n))\|\leq\eps
\]
for all $b\in B$ with $\|b\|\leq 1$.
\end{lemma}
\begin{proof}
Given any $b\in B, \|b\|\leq 1$, $\delta>0$ and $g_0\in K$, we can apply \ref{sequence-lift} and find $n_0\in\IN$ and a neighbourhood $U\ni g_0$ such that
\[
\sup_{k\geq n_0}~\sup_{g\in U}~ \| b(\beta_{g_0}(z_k)-\beta_g(z_k) \| \leq\delta.
\]
By compactness of $K$, we can find $N\in\IN$, elements $\set{g_j}_{j=1}^N\subset K$ and a covering $K\subset\bigcup_{j=1}^N U_j$ with $g_j\in U_j$ and some $n\in\IN$ such that
\[
\sup_{k\geq n}~\sup_{g\in U_j}~ \| b(\beta_{g_j}(z_k)-\beta_g(z_k) \| \leq\delta
\]
for all $j=1,\dots,N$. Since $\phi(A)$ is non-degenerate in $B$, it follows from our assumptions on $z$ that $\|b(\beta_{\infty,g_j}(z)-z)\|\leq\eps$ for all $j=1,\dots,N$. 
By making $n$ larger, if necessary, we can thus assume that
\[
\sup_{k\geq n}~\|b (z_k-\beta_{g_j}(z_k)) \| \leq\eps
\]
for all $j=1,\dots,N$. So if $g\in K$ is arbitrary, then there is some $j$ with $g\in U_j$, and so
\[
\|b(\beta_g(z_k)-z_k)\| \leq  \| b(\beta_{g_j}(z_k)-\beta_g(z_k) \| + \|b (z_k-\beta_{g_j}(z_k)) \| \leq \delta+\eps
\]
for all $k\geq n$. Since $\delta$ and $b$ were chosen arbitrary, this finishes the proof.
\end{proof}

\begin{rem}
\label{approx-fixed-b}
By inserting $b=\eins_B$ in \ref{approx-fixed}, we obtain the following unital variant:

Let $B$ be unital and $z\in B_\infty\cap A'$ an element such that $[g\mapsto \beta_{\infty,g}(z)]$ is continuous.
Let $\eps>0$, $K\subset G$ a compact set and assume that
$\|z-\beta_{\infty,g}(z) \|\leq\eps$ for all $g\in K$. If $(z_n)_n$ is a bounded sequence representing $z$, then
\[
\limsup_{n\to\infty} ~\max_{g\in K}~ \|z_n-\beta_g(z_n)\|\leq\eps.
\]
\end{rem}

\begin{defi}[cf.~{\cite[Section 6]{IzumiMatui10} and \cite[2.3]{GoldsteinIzumi11}}] 
Let $\gamma: G\curvearrowright D$ be an action of a second-countable, locally compact group on a separable, unital \cstar-algebra. We say that the action $\gamma$ (or the system $(D,\gamma)$) has approximately $G$-inner half-flip, if the two equivariant $*$-homomorphisms $\id_D\otimes\eins, \eins\otimes\id_D: (D,\gamma)\to (D\otimes D,\gamma\otimes\gamma)$ are approximately $G$-unitarily equivalent.
\end{defi}

We are now going to prove an equivariant McDuff-type absorption result in the unital case, modelling the well-known McDuff-type results \cite[7.2.2]{Rordam}, \cite[2.3]{TomsWinter07} or \cite[4.11]{Kirchberg04} from the non-equivariant setting. We note that the result given below is folklore for discrete acting groups, and was already used by Izumi-Matui in \cite[6.3]{IzumiMatui10}, Goldstein-Izumi in \cite[2.3]{GoldsteinIzumi11} and Matui-Sato in \cite[4.8]{MatuiSato12_2} and \cite[proof of 4.9]{MatuiSato14}.

\begin{theorem} 
\label{equi mcduff} 
Let $A$ and $D$ be separable, unital \cstar-algebras and $G$ a second-countable, locally compact group. Assume that $(\alpha,u): G\curvearrowright A$ is a cocycle action. Let $\gamma: G\curvearrowright D$ be an action with approximately $G$-inner half-flip. If there exists an equivariant and unital $*$-homomorphism
\[
(D,\gamma) \to (A_{\infty,\alpha}\cap A', \alpha_\infty),
\]
then $(A,\alpha,u)$ is strongly cocycle conjugate to $(D\otimes A, \gamma\otimes\alpha,\eins_D\otimes u)$.
\end{theorem}
\begin{proof}
Consider the canonical unital $*$-homomorphisms
\[
(A_\infty\cap A')\otimes_{\max} A \to A_\infty,\quad x\otimes a\mapsto x\cdot a.
\]
Note that this is $(\alpha_\infty\otimes\alpha)$-to-$\alpha_\infty$ equivariant. This induces another unital and equivariant $*$-homomorphism (disregarding the cocycles for now)
\[
\big( D\otimes (A_{\infty,\alpha}\cap A')\otimes_{\max} A, \gamma\otimes\alpha_\infty\otimes\alpha \big) \to \big( (D\otimes A)_{\infty,\gamma\otimes\alpha}, (\gamma\otimes\alpha)_\infty \big).
\]
via $d\otimes x\otimes a\mapsto d\otimes (x\cdot a)\in D\otimes (A_{\infty,\alpha}) \subset (D\otimes A)_{\infty,\gamma\otimes\alpha}$.
By our assumption, it follows that there is a unital and equivariant $*$-homomorphism
\[
\Phi: \big( D\otimes D\otimes A, \gamma\otimes\gamma\otimes\alpha) \to \big( (D\otimes A)_{\infty,\gamma\otimes\alpha}, (\gamma\otimes\alpha)_\infty \big)
\]
with $\Phi(d\otimes\eins_D\otimes a)=d\otimes a$ for all $a\in A$ and $d\in D$, and with $\Phi(\eins_D\otimes D\otimes A)\subset\eins_{D}\otimes A_\infty$. Since $(D,\gamma)$ has approximately $G$-inner half-flip, choose a sequence $v_n\in D\otimes D$ with $\dst\max_{g\in K}\|v_n-(\gamma\otimes\gamma)_g(v_n)\|\to 0$ for every compact $K\subset G$, and $v_n^*(d\otimes\eins_D)v_n \to \eins_D\otimes d$ for all $d\in D$.
The sequence of unitaries $u_n=\Phi(v_n\otimes\eins_A)\in (D\otimes A)_{\infty,\gamma\otimes\alpha}$ then satisfies
\begin{itemize}
\item $[u_n, \eins_D\otimes a] = \Phi([v_n\otimes\eins_A, \eins_{D\otimes D}\otimes a])=0$ for all $a\in A$;
\item $u_n^*(d\otimes a)u_n = \Phi\big( (v_n\otimes\eins_A)(d\otimes\eins_D\otimes a)(v_n^*\otimes\eins_A) \big)\to \Phi(\eins\otimes d\otimes a) \in \eins_D\otimes A_\infty$ for all $a\in A$ and $d\in D$;
\item $\|u_n-(\gamma\otimes\alpha)_{\infty,g}(u_n)\| = \|\Phi\big( (v_n-(\gamma\otimes\gamma)_g(v_n))\otimes\eins_A \big)\| \to 0$ for all $g\in G$, and uniformly on compact sets.
\end{itemize}
Lifting each $u_n$ to a sequence of unitaries $(u_k^{(n)})_k\in\ell^\infty(\IN,D\otimes A)$ and applying a diagonal sequence argument using \ref{approx-fixed-b}, we can find a sequence of unitaries $z_n=u_{k_n}^{(n)}\in D\otimes A$ satisfying
\begin{itemize}
\item $[z_n, \eins_D\otimes a]\to 0$ for all $a\in A$;
\item $\dist( z_n^*(d\otimes a)z_n, \eins\otimes A) \to 0$ for all $a\in A$ and $d\in D$;
\item $\dst\max_{g\in K}\|z_n-(\gamma\otimes\alpha)_g(z_n)\|\to 0$ for all compact $K\subset G$.
\end{itemize}
In particular, we have met the conditions of \ref{eq:oneside-a}, \ref{eq:oneside-b} and \ref{eq:oneside-c} for the equivariant embedding $\eins_D\otimes\id_A: (A,\alpha,u)\to (D\otimes A, \gamma\otimes\alpha, \eins_D\otimes u)$. This finishes the proof.
\end{proof}

\begin{reme}
In the non-equivariant situation, a suitable replacement for $A_\infty\cap A'$ in \ref{equi mcduff} can be $\CM(A)_\infty\cap A'$ to cover the non-unital case, see \cite[7.2.2]{Rordam}. In the equivariant case, the above result \ref{equi mcduff} can be shown to generalize with using this \cstar-algebra, but only if $\alpha$ is assumed to be a genuine action. If $(\alpha,u): G\curvearrowright A$ is a cocycle action with non-trivial cocycles, then $\CM(A)_\infty\cap A'$ might not be the right object to consider, since $\alpha$ does not necessarily induce a genuine action on this \cstar-algebra.
The more current (non-equivariant) McDuff type theorems for strongly self-absorbing \cstar-algebras \cite[2.3]{TomsWinter07} and \cite[4.1]{Kirchberg04} shift away from considering $\CM(A)_\infty\cap A'$ in the non-unital case, in order to instead consider (either implicitely or explicitely) its quotient \cstar-algebra $F_\infty(A)$. As seen in \ref{induced-maps-and-unitaries}, a cocycle action on $A$ gives rise to a genuine action on $F_\infty(A)$ in a natural way, thus indicating that $F_\infty(A)$ is the more natural object to look at.

Both in \cite{TomsWinter07} and \cite{Kirchberg04}, the cost of working with $F_\infty(A)$ in the non-unital case is that one has to use a strengthened version of approximate inner half-flip, which essentially gives rise to all the $K_1$-injectivity assumptions in both sources (however, we note that $K_1$-injectivity later turned out to be redundant by virtue of \cite{Winter11}). In the next section, we are going to generalize this approach to the equivariant situation, and prove an optimal McDuff-type theorem by considering equivariant embeddings into $F_\infty(A)$.
\end{reme}


\section{Strongly self-absorbing \cstar-dynamical systems}
\noindent
In this section, we will introduce strongly self-absorbing actions and generalize some partial results from Toms-Winter's seminal work \cite{TomsWinter07} on strongly self-absorbing \cstar-algebras.

\begin{defi}
\label{ssa system}
Let $\CD$ be a separable, unital \cstar-algebra and $G$ a second-countable, locally compact group. Let $\gamma: G\curvearrowright\CD$ be a point-norm continuous action. We say that $(\CD,\gamma)$ is a strongly self-absorbing \cstar-dynamical system, or that $\gamma$ is a strongly self-absorbing action, if the equivariant first-factor embedding
\[
\id_\CD\otimes\eins_\CD: (\CD,\gamma)\to (\CD\otimes\CD,\gamma\otimes\gamma)
\]
is approximately $G$-unitarily equivalent to an isomorphism.
\end{defi}

\begin{example}
Let $G$ be a second-countable, locally compact group and $\CD$ a strongly self-absorbing \cstar-algebra. Then the trivial $G$-action on $\CD$ is a strongly self-absorbing action.
\end{example}

\begin{reme}
Trivially, any two equivariant $*$-homomorphisms, which are approximately $G$-unitarily equivalent, are also approximately unitarily equivalent as ordinary $*$-homomorphisms. So if $(\CD,\gamma)$ is a strongly self-absorbing \cstar-dynamical system, then $\CD$ must be a strongly self-absorbing \cstar-algebra, see \cite{TomsWinter07}.
\end{reme}

We now make some important observations concerning strongly self-abs\-orbing actions on the one hand, and actions with approximately $G$-inner half-flip on the other hand. This mimics the situation in the non-equivariant case \cite{TomsWinter07}: 

\begin{prop}[cf.~{\cite[1.9]{TomsWinter07}}]
\label{aih-ssa}
Let $D$ be a separable, unital \cstar-algebra. Let $\gamma: G\curvearrowright D$ be an action with approximately $G$-inner half-flip. Then: 
\begin{enumerate}[label=\textup{(\roman*)},leftmargin=*]
\item \label{aih-ssa:1}
Let $\alpha: G\curvearrowright A$ be on action on a unital \cstar-algebra.
Suppose that there exists an equivariant unital $*$-homomorphism from $(D,\gamma)$ to $(A_{\infty,\alpha}\cap A',\alpha_\infty)$.
Then all unital equivariant $*$-homomorphisms from $(D,\gamma)$ to $(A,\alpha)$ are mutually approximately $G$-unitarily equivalent.
\item \label{aih-ssa:2}
The infinite tensor power $(\bigotimes_\IN D, \bigotimes_\IN\gamma)$ is strongly self-absorbing.
\end{enumerate}
\end{prop}
\begin{proof}
\ref{aih-ssa:1}:
Let $\phi_1$ and $\phi_2$ be two such $*$-homomorphisms.
Using the underlying assumption, we may find a unital equivariant $*$-homomorphism $\kappa: (D,\gamma)\to (A_{\infty,\alpha}\cap A',\alpha_\infty)$.
In particular, the range of $\kappa$ pointwise commutes with elements in $\phi_1(D)\cup\phi_2(D)$.
Consider the following diagram of equivariant $*$-homomorphisms, keeping in mind that the upper and lower triangles commute:
\[
\xymatrix{
&& D\otimes D \ar[rrd]^{\phi_1\times\kappa} && \\
D \ar[rru]^{\id\otimes\eins} \ar[rrd]_{\id\otimes\eins} \ar@/^0.70pc/[rrrr]^{\phi_1} \ar@/_0.70pc/[rrrr]_{\phi_2} &&&& A_{\infty,\alpha} \\
&& D\otimes D \ar[rru]_{\phi_2\times\kappa} &&
}
\]
As $(D,\gamma)$ has approximately $G$-inner half-flip, we see that both $\phi_1$ and $\phi_2$ are approximately $G$-unitarily equivalent to $\kappa$ as equivariant $*$-homomorphisms into $A_{\infty,\alpha}$.
By \ref{approx-fixed-b}, this implies that $\phi_1$ and $\phi_2$ are indeed approximately $G$-unitarily equivalent as equivariant $*$-homomorphisms into $A$.

\ref{aih-ssa:2}: Set $(\CD,\delta)=(\bigotimes_\IN D, \bigotimes_\IN\gamma)$.
It is easy to see that $(\CD,\delta)$ also has approximately $G$-inner half-flip.
Clearly it admits an approximately central and equivariant embedding of $(D,\gamma)$, and therefore also of $(\CD,\delta)$ itself via a standard reindexation trick.
Hence it fits the assumptions in \ref{aih-ssa:1} in place of both $(D,\gamma)$ and $(A,\alpha)$.
It is trivial that there exists an equivariant isomorphism $\phi: (\CD,\delta)\to (\CD\otimes\CD,\delta\otimes\delta)$.
If we apply \ref{aih-ssa:1}, it follows that $\phi$ is approximately $G$-unitarily equivalent to the first-factor embedding $\id_\CD\otimes\eins_\CD$, showing the claim.
\end{proof}

\begin{prop}[cf.~{\cite[1.10(ii), 1.5, 1.12]{TomsWinter07}}]
\label{ssa-facts}
Let $(\CD,\gamma)$ be a strongly self-absorbing \cstar-dynamical system. Then:
\begin{enumerate}[label=\textup{(\roman*)},leftmargin=*]
\item There exists a unital and equivariant $*$-homomorphism from $(\CD,\gamma)$ to 
$(\CD_{\infty,\gamma}\cap\CD', \gamma_\infty)$. \label{ssa-facts:1}
\item $(\CD,\gamma)$ has approximately $G$-inner half-flip. \label{ssa-facts:2}
\item Let $(A,\alpha)$ be a \cstar-dynamical system on a unital \cstar-algebra such that $(A,\alpha)$ is cocycle conjugate to $(A\otimes\CD,\alpha\otimes\gamma)$. Then any two equivariant, unital $*$-homomorphisms from $(\CD,\gamma)$ to $(A,\alpha)$ are approximately $G$-unitarily equivalent. \label{ssa-facts:3}
\end{enumerate}
\end{prop}
\begin{proof}
Let us show \ref{ssa-facts:1}.
Let $\phi: (\CD,\gamma)\to (\CD\otimes\CD,\gamma\otimes\gamma)$ be an equivariant isomorphism and $v_n\in\CU(\CD\otimes\CD)$ a sequence of unitaries such that
\[
\lim_{n\to\infty} \max_{g\in K} \|v_n-(\gamma\otimes\gamma)_g(v_n)\| = 0
\]
for every compact set $K\subseteq G$, and
\[
\phi(x)=\lim_{n\to\infty} v_n(x\otimes\eins)v_n^*
\]
for all $x\in\CD$.
The sequence of $*$-homomorphisms $\psi_n=\phi^{-1}\circ\ad(v_n^*)\circ(\eins_\CD\otimes\id_\CD)$ is clearly approximately equivariant as $n\to\infty$.
We also claim that it is an approximately central embedding.
For any $x,y\in\CD$, we observe
\[
\begin{array}{ccl}
\dst \lim_{k\to\infty} [x,\psi_k(y)] &=& \dst \lim_{k\to\infty} [\phi^{-1}(\phi(x)), \psi_k(y) ] \\
&=& \dst \lim_{k\to\infty} \phi^{-1}\big( v_n^*\cdot [x\otimes\eins, \eins\otimes y] \cdot v_n \big) \\
&=& 0.
\end{array}
\]
Hence the sequence $\psi_n$ induces a unital equivariant $*$-homomorphism $\psi=[(\psi_n)_n]: (\CD,\gamma)\to(\CD_{\infty,\gamma}\cap\CD', \gamma_\infty)$.

In order to show \ref{ssa-facts:2}, one carries out Toms-Winter's original proof from the non-equivariant case \cite[1.5]{TomsWinter07} verbatim, only replacing approximate unitary equivalence by approximate $G$-unitary equivalence.

Let us show \ref{ssa-facts:3}. 
By the previous two parts, it follows that $(\CD,\gamma)$ has approximately $G$-inner half-flip, and there is an equivariant unital $*$-homomorphism from $(\CD,\gamma)$ into the central sequence algebra of $A$.
Hence the claim follows directly from \ref{aih-ssa}\ref{aih-ssa:1}.
\end{proof}

\begin{prop}[cf.~{\cite[1.13]{TomsWinter07}}]
\label{commutator-unitaries}
Let $(\CD,\gamma)$ be a strongly self-absorbing \cstar-dynamical system. Then there exist sequences of unitaries $u_n,v_n\in\CU(\CD\otimes\CD)$ satisfying
\[
\max_{g\in K} \Big( \|u_n-(\gamma\otimes\gamma)_g(u_n)\|+\|v_n-(\gamma\otimes\gamma)_g(v_n)\| \Big) \stackrel{n\to\infty}{\longrightarrow} 0
\]
for every compact set $K\subset G$ and
\[
\ad(u_nv_nu_n^*v_n^*)(d\otimes\eins_\CD) \stackrel{n\to\infty}{\longrightarrow} \eins_\CD\otimes d
\]
for all $d\in\CD$.
\end{prop}
\begin{proof}
This arises as a modification of the proof of \ref{aih-ssa}\ref{aih-ssa:1}.
We consider $(A,\alpha)=(\CD\otimes\CD,\gamma\otimes\gamma)$ and the two unital equivariant $*$-homomorphisms $\phi_1=\id_\CD\otimes\eins_\CD$ and $\phi_2=\eins_\CD\otimes\id_\CD$ from $(\CD,\gamma)$ to $(A,\alpha)$.
Exactly as in the proof of \ref{aih-ssa}\ref{aih-ssa:1}, we pick a $*$-homomorphism $\kappa: (\CD,\gamma)\to (A_{\infty,\alpha}\cap A',\alpha_\infty)$ and arrive at the following diagram:
\[
\xymatrix{
&& \CD\otimes \CD \ar[rrd]^{\phi_1\times\kappa} && \\
\CD \ar[rru]^{\id\otimes\eins} \ar[rrd]_{\id\otimes\eins} \ar@/^0.70pc/[rrrr]^{\phi_1} \ar@/_0.70pc/[rrrr]_{\phi_2} &&&& A_{\infty,\alpha} \\
&& \CD\otimes \CD \ar[rru]_{\phi_2\times\kappa} &&
}
\]
We already know that $(\CD,\gamma)$ has approximately $G$-inner half-flip, so choose a sequence of unitaries $u_n\in\CU(A)$ with $\lim_{n\to\infty} \max_{g\in K} \|u_n-\alpha_g(u_n)\|=0$ for every compact subset $K\subseteq G$, and $\lim_{n\to\infty} u_n\phi_1(x)u_n^*=\phi_2(x)$ for all $x\in\CD$. 
Set $v_n=(\phi_1\times\kappa)(u_n)^*$.
Then, if $x\in\CD$ is an element and $n$ is a large enough number, we have
\[
\begin{array}{ccl}
\ad(u_nv_nu_n^*v_n)(\phi_1(x)) &=& \ad(u_nv_nu_n^*)\Big( (\phi_1\times\kappa)(u_n(x\otimes\eins)u_n^*) \Big) \\
&\approx& \ad(u_nv_nu_n^*)(\kappa(x)) \\
&=& \ad(u_nv_n)(\kappa(x)) \\
&=& \ad(u_n)\Big( (\phi_1\times\kappa)(u_n^*(\eins\otimes x)u_n) \Big) \\
&\approx& \ad(u_n)(\phi_1(x)) \\
&\approx& \phi_2(x).
\end{array}
\]
In this construction, the sequence $v_n$ is in $A_{\infty,\alpha}$.
However, we may represent each of them by a sequence of unitaries.
Due to \ref{approx-fixed-b}, we may run a standard diagonal sequence argument a arrive at a sequence of unitaries $u_n,v_n\in\CU(A)$ satisfying the desired properties as stated in the claim.
\end{proof}

In particular, the means that in strongly self-absorbing \cstar-dynamical systems, the approximate $G$-inner half-flip may be chosen to be implemented by unitaries representing the zero-class in $K_1$, similarly as in \cite{TomsWinter07}. By $K_1$-injectivity of strongly self-absorbing \cstar-algebras, which follows from \cite{Jiang, Winter11}, we get the following consequence:

\begin{cor}[cf.~\cite{Jiang} and {\cite[3.1, 3.3]{Winter11}}]
\label{1-hom-unitaries}
Let $(\CD,\gamma)$ be a strongly self-absorbing \cstar-dynamical system. Then there exists a sequence of unitaries $v_n\in\CU_0(\CD\otimes\CD)$ satisfying
\[
\max_{g\in K}\|v_n-(\gamma\otimes\gamma)_g(v_n)\| \stackrel{n\to\infty}{\longrightarrow} 0
\]
for every compact set $K\subset G$ and
\[
v_n(d\otimes\eins_\CD)v_n^* \stackrel{n\to\infty}{\longrightarrow} \eins_\CD\otimes d
\]
for all $d\in\CD$.
\end{cor}

The following constitutes the main result of this paper:

\begin{theorem}
\label{equi-McDuff}
Let $G$ be a second-countable, locally compact group. Let $A$ be a separable \cstar-algebra and $(\alpha,u): G\curvearrowright A$ a cocycle action. Let $\CD$ be a separable, unital \cstar-algebra and $\gamma: G\curvearrowright\CD$ an action such that $(\CD,\gamma)$ is strongly self-absorbing.
The following are equivalent:
\begin{enumerate}[label=\textup{(\roman*)},leftmargin=*] 
\item $(A,\alpha,u)$ is strongly cocycle conjugate to $(A\otimes\CD,\alpha\otimes\gamma,u\otimes\eins_\CD)$. \label{equi-McDuff1}
\item $(A,\alpha,u)$ is cocycle conjugate to $(A\otimes\CD,\alpha\otimes\gamma,u\otimes\eins_\CD)$. \label{equi-McDuff2}
\item There exists a unital, equivariant $*$-homomorphism from $(\CD,\gamma)$ to $\big( F_{\infty,\alpha}(A), \tilde{\alpha}_\infty \big)$. \label{equi-McDuff3}
\item The first-factor embedding $\id_A\otimes\eins_\CD: (A,\alpha,u)\to (A\otimes\CD,\alpha\otimes\gamma,u\otimes\eins_\CD)$ is approximately unitarily equivalent to an isomorphism inducing strong cocycle conjugacy. \label{equi-McDuff4}
\end{enumerate}
\end{theorem}
\begin{proof}
The implications from \ref{equi-McDuff4} to \ref{equi-McDuff1} and from \ref{equi-McDuff1} to \ref{equi-McDuff2} are trivial. The implication from \ref{equi-McDuff2} to \ref{equi-McDuff3} follows from \ref{ssa-facts}\ref{ssa-facts:1} with \ref{cont-central-seq}. So we need to verify the implication from \ref{equi-McDuff3} to \ref{equi-McDuff4}.

Keeping in mind \ref{F(A)-enhanced}, we have a natural isomorphism
\[
F(\eins_\CD\otimes A,(\CD\otimes A)_\infty)\cong F(\eins_\CD\otimes A, \big( (\CD\otimes A)^\sim \big)_\infty ).
\]
Denote by $\pi: \big( (\CD\otimes A)^\sim \big)_\infty\cap (\eins_\CD\otimes A)' \to F(\eins_\CD\otimes A,(\CD\otimes A)_\infty)$ the canonical surjection. Note that by assumption, we have an equivariant, unital $*$-homomorphism from $(\CD,\gamma)$ to $F_{\infty,\alpha}(A)$. Consider the canonical inclusions 
\[
F_{\infty}(A) ,~\CD ~\subset~ F(\eins_\CD\otimes A,(\CD\otimes A)_\infty),
\]
which define commuting \cstar-subalgebras. Since these inclusions are natural, they are equivariant with respect to the induced actions of $\alpha$, $\gamma$ and $\gamma\otimes\alpha$.
By assumption, it follows we have a unital and equivariant $*$-homomorphism
\[
\phi: (\CD\otimes\CD,\gamma\otimes\gamma)\to \big( F_{\gamma\otimes\alpha}(\eins_\CD\otimes A,(\CD\otimes A)_\infty) , (\gamma\otimes\alpha)^\sim_\infty \big)
\]
satisfying $\phi(d\otimes\eins_D)\cdot (\eins_\CD\otimes a)=d\otimes a$ and $\phi(\eins_\CD\otimes d)\cdot (\eins_\CD\otimes a)\in\eins_{\CD}\otimes A_\infty$ for all $a\in A$ and $d\in\CD$. (Note that forming these products makes sense because of \ref{kirchberg-projection}.) 

Now let $\eps>0$, $F_1\fin\CD$, $F_2\fin A$ and $K\subset G$ compact. Without loss of generality, assume that $F_1$ and $F_2$ consist of contractions.
Since $(\CD,\gamma)$ has approximately $G$-inner half-flip, choose a unitary $v\in \CD\otimes\CD$ with 
\[
\max_{g\in K}\|v-(\gamma\otimes\gamma)_g(v)\|\leq\eps \quad\text{and}\quad v^*(d\otimes\eins_\CD)v =_\eps \eins_\CD\otimes d
\]
for all $d\in F_1$. Moreover, because of \ref{1-hom-unitaries}, we may assume that $v$ is homotopic to $\eins_{\CD\otimes\CD}$ in $\CU(\CD\otimes\CD)$.

The unitary $u=\phi(v)\in F_{\gamma\otimes\alpha}(\eins_\CD\otimes A,(\CD\otimes A)_\infty)$ then satisfies
\[
u^*(d\otimes a)u = \phi(v(d\otimes\eins_\CD)v^*)\cdot (\eins_\CD\otimes a) =_\eps \phi(\eins_\CD\otimes d)\cdot (\eins_\CD\otimes a) \in \eins_\CD\otimes A_\infty
\] 
for all $a\in A$ with $\|a\|\leq 1$ and $d\in F_1$, and moreover
\[
\|u-(\gamma\otimes\alpha)^\sim_{\infty,g}(u)\| \leq \|v-(\gamma\otimes\gamma)_g(v)\| \leq\eps
\] 
for all $g\in K$.

Let $z\in \big( (\CD\otimes A)^\sim \big)_\infty\cap (\eins_\CD\otimes A)'$ be an element with $u=\pi(z)$. Since we have chosen $v$ to be homotopic to $\eins_{\CD\otimes\CD}$, $u$ is also homotopic to $\eins$ in $F\big( \eins_\CD\otimes A, (\CD\otimes A)_\infty \big)$, and therefore we can choose the lift $z=[(z_n)_n]$ to be represented by a sequence of unitaries. 
Note that $u$ is a continuous element with respect to $(\gamma\otimes\alpha)^\sim_\infty$, so it follows that $[g\mapsto (\eins_\CD\otimes a)\cdot (\gamma\otimes\alpha)_{\infty,g}(z)]$ is a continuous map on $G$ for every $a\in A$. Moreover, because of the properties of $u$, we have
\[
\dist(z^*(d\otimes a)z, \eins_\CD\otimes A_\infty)\leq\eps
\] 
for all $a\in A$ with $\|a\|\leq 1$ and $d\in F_1$, and moreover
\[
\|(\eins_\CD\otimes a)\big( z-(\gamma\otimes\alpha)_{\infty,g}(z) \big)\| \leq\eps
\] 
for all $g\in K$ and $a\in A$ with $\|a\|\leq 1$.
  
Let $z_n\in (\CD\otimes A)^\sim$ be a bounded sequence representing $z$. Using \ref{approx-fixed}, there is some $n$ with
\begin{itemize}
\item $\|[z_n, \eins_\CD\otimes a]\|\leq\eps$ for all $a\in F_2$;
\item $\dist( z_n^*(d\otimes a)z_n, \eins_\CD\otimes A)\leq 2\eps$ for all $d\in F_1$ and $a\in F_2$;
\item $\dst\max_{g\in K}\|z_n-(\gamma\otimes\alpha)_g(z_n)\| \leq 2\eps$.
\end{itemize}
So we have met the conditions of \ref{eq:oneside-a}, \ref{eq:oneside-b} and \ref{eq:oneside-c} for the equivariant embedding $\eins_\CD\otimes\id_A: (A,\alpha,u)\to (D\otimes A, \gamma\otimes\alpha, \eins_D\otimes u)$. This finishes the proof. 
\end{proof}

\begin{cor}
Let $G$ be a second-countable, locally compact group. Let $A$ be a separable \cstar-algebra and $(\alpha,u): G\curvearrowright A$ a cocycle action. Let $\CD$ be a strongly self-absorbing \cstar-algebra. The following are equivalent:
\begin{enumerate}[label=\textup{(\roman*)},leftmargin=*] 
\item $(A,\alpha,u)$ is strongly cocycle conjugate to $(A\otimes\CD,\alpha\otimes\id_\CD,u\otimes\eins_\CD)$ 
\item $(A,\alpha,u)$ is cocycle conjugate to $(A\otimes\CD,\alpha\otimes\id_\CD,u\otimes\eins_\CD)$ 
\item There exists a unital, equivariant $*$-homomorphism from $\CD$ to the fixed point algebra $F_{\infty,\alpha}(A)^{\tilde{\alpha}_\infty}$
\item The first-factor embedding $\id_A\otimes\eins_\CD: (A,\alpha,u)\to (A\otimes\CD,\alpha\otimes\id_\CD,u\otimes\eins_\CD)$ is approximately unitarily equivalent to an isomorphism inducing strong cocycle conjugacy.
\end{enumerate}
\end{cor}

We finish this section with an example that illustrates how for non-compact acting groups, our main result \ref{equi-McDuff} cannot be strengthened to characterize tensorial absorption with respect to conjugacy:

\begin{example}
Let $\theta\in (0,1)$ be an irrational number. Consider the Cuntz algebra $\CO_2$ (see \cite{Cuntz77}) with the two canonical generators $s_1,s_2\in\CO_2$. Let the automorphism $\alpha\in\Aut(\CO_2)$ be given by $\alpha(s_j)=e^{2\pi i\theta}s_j$ for $j=1,2$. Then the associated $\IZ$-action $\alpha: \IZ\curvearrowright\CO_2$ arises as the composition of the irrational rotation embedding $\IZ\to\IT$ as a dense subgroup, with the natural gauge action of $\IT$ on $\CO_2$ given by $z.s_j=z\cdot s_j$ for $j=1,2$. So we have $\CO_2^\alpha\cong M_{2^\infty}$ because the fixed-point algebra of the gauge action is well-known to be the CAR algebra.
On the other hand, $\alpha$ is quasi-free and faithful, so in particular pointwise outer, see \cite[Section 2]{Evans80}. Since any endomorphism on $\CO_2$ is asymptotically inner and $\CO_2\cong\CO_2\otimes\CO_2$, it follows from Nakamura's uniqueness result \cite[Theorem 9]{Nakamura00} that $\alpha$ is cocycle conjugate to $\alpha\otimes\id_{\CO_2}$. Since the fixed-point algebra of $\alpha$ is stably finite, these two actions clearly cannot be conjugate.
\end{example}


\section{Semi-strongly self-absorbing actions}
\noindent
As we have seen in the previous section, strong cocycle conjugacy is in general a strictly weaker relation than conjugacy. In the non-equivariant context, this distinction naturally plays no role. So we see that this aspect adds some complexity to the equivariant situation. For instance, it follows from several results from the literature (see the fifth section for details) that many actions on a strongly self-absorbing \cstar-algebra can be strongly cocycle conjugate (but not conjugate) to a strongly self-absorbing action. This prompts the question whether one can expand the tensorial absorption theorem from the previous section to a larger class of actions. In this section, we will therefore study \cstar-dynamical systems that are, in a sense, almost strongly self-absorbing:

\begin{defi}
Let $G$ be a second-countable, locally compact group and $\CD$ a separable, unital \cstar-algebra. We call an action $\gamma: G\curvearrowright\CD$ semi-strongly self-absorbing, if it is strongly cocycle conjugate to a strongly self-absorbing action.
\end{defi}

We will study some properties that semi-strongly self-absorbing actions inherit from strongly self-absorbing actions. However, we first have to make a few general observations.

\begin{lemma}
\label{scc-approx-conjugate}
Let $G$ be a second-countable, locally compact group and $A$ and $B$ two separable, unital \cstar-algebras. Suppose that $\alpha: G\curvearrowright A$ and $\beta: G\curvearrowright B$ are two actions that are strongly cocycle conjugate. Then for every $\eps>0$ and compact set $K\subset G$, there exists an isomorphism $\phi: A\to B$ with $\max_{g\in K} \|\beta_g\circ\phi-\phi\circ\alpha_g\|\leq\eps$.
\end{lemma}
\begin{proof}
Using that $\alpha$ and $\beta$ are strongly cocycle conjugate, we find an isomorphism $\psi: A\to B$ and a sequence of unitaries $x_n\in\CU(B)$ such that the functions $[g\mapsto x_n\beta_g(x_n^*)]$ on $G$ converge uniformly on compact sets as $n$ tends to infinity, such that with $w_g=\lim_{n\to\infty} x_n\beta_g(x_n^*)$, we have
\[
\ad(w_g)\circ\beta_g\circ\psi = \psi\circ\alpha_g\quad\text{for all}~g\in G.
\]
Applying $\ad(x_n^*)$ to both sides of this equation, we obtain
\[
\ad(x_n^*)\circ\psi\circ\alpha_g = \ad(x_n^*w_g)\circ\beta_g\circ\psi = \ad(x_n^*w_g\beta_g(x_n))\circ\beta_g\circ\ad(x_n^*)\circ\psi.
\]
Let $\eps>0$ and $K\subset G$ be given. By the definition of the cocycle $w_g$, there is some $n$ such that $\max_{g\in K} \|x_n^*w_g\beta_g(x_n)-\eins_B\|\leq\eps/2$. Then the map $\phi=\ad(x_n^*)\circ\psi$ has the properties that we are looking for.
\end{proof}

\begin{lemma}
\label{scc-aih}
Let $G$ be a second-countable, locally compact group and $A$ and $B$ two separable, unital \cstar-algebras. Suppose that $\alpha: G\curvearrowright A$ and $\beta: G\curvearrowright B$ are two actions that are strongly cocycle conjugate. If $\beta$ has approximately $G$-inner half-flip, then so does $\alpha$.
\end{lemma}
\begin{proof}
Let $\eps>0$, $F\fin A$ and a compact set $K\subset G$ be given. Applying \ref{scc-approx-conjugate}, we find a sequence of isomorphisms $\phi_n: A\to B$ with 
\[
\max_{g\in K}\|\beta_g\circ\phi_n-\phi_n\circ\alpha_g\| \stackrel{n\to\infty}{\longrightarrow} 0.
\]
It follows that
\[
\max_{g\in K}\|(\beta\otimes\beta)_g\circ(\phi_n\otimes\phi_n)-(\phi_n\otimes\phi_n)\circ(\alpha\otimes\alpha)_g\| \leq\eps
\]
for some large enough $n$.
Choose a unitary $v\in B\otimes B$ with 
\[
\max_{g\in K} \|v-(\beta\otimes\beta)_g(v)\|\leq\eps
\]
and
\[
v(\phi_n(a)\otimes\eins_B)v^* =_\eps \eins_B\otimes\phi_n(a)\quad\text{for all}~a\in F.
\]
For the unitary $u=(\phi_n\otimes\phi_n)^{-1}(v)$, it follows that
\[
\begin{array}{rl}
\multicolumn{2}{l}{ \dst\max_{g\in K}\|u-(\alpha\otimes\alpha)_g(u)\| } \\
\hspace{10mm}=& \dst\max_{g\in K} \|v-(\phi_n\otimes\phi_n)\circ(\alpha\otimes\alpha)_g\circ(\phi_n\otimes\phi_n)^{-1}(v)\| \\
\leq& \dst\eps+\max_{g\in K} \|v-(\beta\otimes\beta)_g(v)\| ~\leq~ 2\eps. 
\end{array}
\]
Moreover, we have
\[
u(a\otimes\eins_A)u^* =_\eps \eins_A\otimes a\quad\text{for all}~a\in F
\]
by choice of $v$. This finishes the proof.
\end{proof}

The following can be seen as a stronger variant of \ref{approx-fixed}:

\begin{lemma}
\label{uniform-equivariance}
Let $G$ be a second-countable, locally compact group and $A$ a separable \cstar-algebra. Suppose that $(\alpha,u): G\curvearrowright A$ is a cocycle action. 
Let $x\in F_{\infty,\alpha}(A)$ be a continuous element with respect to $\alpha$. Let $(x_n)_n\in\ell^\infty(\IN,A)$ be a representing sequence of $x$. Then for every $\eps>0$, $a\in A_{1,+}$ and compact set $K\subset G$, there is $n_0\in\IN$ such that
\[
\sup_{n\geq n_0}~ \| \big( \alpha_g(x_n)-x_n \big)a \| \leq \|\tilde{\alpha}_{\infty,g}(x)-x\|+\eps
\]
for all $g\in K$.
\end{lemma}
\begin{proof}
By definition, we have
\begin{equation} \label{eq3:1}
\|\tilde{\alpha}_{\infty,g}(x)-x\| = \sup_{a\in A_{1,+}}~\limsup_{n\to\infty}~\| \big( \alpha_g(x_n)-x_n \big)a \|
\end{equation}
for all $g\in G$. 

Let $\tilde{x}\in A_\infty\cap A'$ be an element with $x=\tilde{x}+\ann(A,A_\infty)$. Then since $x$ is $\tilde{\alpha}_{\infty}$-continuous, it follows that the assignment $g\mapsto \alpha_{\infty,g}(\tilde{x})\cdot a$ is continuous on $G$.

Let $\eps>0$, $K\subset G$ a compact set, and let $a\in A_{1,+}$ be fixed. Let $g_0\in K$.
 Applying \ref{sequence-lift}, choose an open neighbourhood $U\ni g_0$ and $n_0'\in\IN$ with
\begin{equation} \label{eq3:2}
\sup_{n\geq n_0'}~\sup_{g\in U} \|\big( \alpha_g(x_n)-\alpha_{g_0}(x_n) \big)a\| \leq \eps
\end{equation}
and
\begin{equation} \label{eq3:3}
\sup_{g\in U} \|\tilde{\alpha}_{\infty,g}(x)-\tilde{\alpha}_{\infty,g_0}(x)\|\leq\eps.
\end{equation}
Since $g_0\in K$ was arbitrary, we thus get an open cover. By compactness of $K$, we find $g_1,\dots,g_N\in K$ and an open covering $K\subset\bigcup_{i=1}^N U_i$ and $n_0\in\IN$, such that for every $i$, we have $g_i\in U_i$ and $U_i$ satisifes \eqref{eq3:2} with $n_0$ in place of $n_0'$, and \eqref{eq3:3}. 
Using \eqref{eq3:1}, we can further enlarge $n_0$ if necessary and assume that
\begin{equation} \label{eq3:4}
\sup_{n\geq n_0}~\| \big( \alpha_{g_i}(x_n)-x_n \big)a \| \leq \|\tilde{\alpha}_{\infty,g_i}(x)-x\|+\eps
\end{equation}
for all $i=1,\dots,N$.
It follows for all $g\in K$ that
\[
\begin{array}{cl}
\multicolumn{2}{l}{ \dst\sup_{n\geq n_0} \| \big( \alpha_g(x_n)-x_n \big)a \| } \\

\leq & \dst \min_{1\leq i\leq N}~\sup_{n\geq n_0}~\| \big( \alpha_g(x_n)-\alpha_{g_i}(x_n) \big)a\| + \| \big( \alpha_{g_i}(x_n)-x_n \big)a \| \\

\stackrel{\eqref{eq3:4}}{\leq}& \eps+\dst \min_{1\leq i\leq N}~\sup_{n\geq n_0}~\| \big( \alpha_g(x_n)-\alpha_{g_i}(x_n) \big)a\| + \| \tilde{\alpha}_{\infty, g_i}(x)-x \| \\
\leq & \eps+\dst \min_{1\leq i\leq N}~\sup_{n\geq n_0}~\| \big( \alpha_g(x_n)-\alpha_{g_i}(x_n) \big)a\| + \| \tilde{\alpha}_{\infty, g_i}(x)-\tilde{\alpha}_{\infty, g}(x) \| \\
& + \|\tilde{\alpha}_{\infty, g}(x)-x\| \\
\stackrel{\eqref{eq3:2}, \eqref{eq3:3}}{\leq} & 3\eps+\|\tilde{\alpha}_{\infty,g}(x)-x\|.
\end{array}
\]
This finishes the proof.
\end{proof}

\begin{lemma}
\label{scc-homo-ex}
Let $G$ be a second-countable, locally compact group. Let $A$ be a separable \cstar-algebra and $B$ a separable, unital \cstar-algebra. Suppose that $(\alpha,u): G\curvearrowright A$ is a cocycle action, and that $\beta,\gamma: G\curvearrowright B$ are two actions that are strongly cocycle conjugate.
Suppose that there exists a unital and equivariant $*$-homomorphism from $(B,\beta)$ to $(F_{\infty,\alpha}(A),\tilde{\alpha}_\infty)$. Then there exists a unital and equivariant $*$-homomorphism from $(B,\gamma)$ to $(F_{\infty,\alpha}(A),\tilde{\alpha}_\infty)$.
\end{lemma}
\begin{proof}
Let $\kappa: (B,\beta)\to (F_{\infty,\alpha}(A),\tilde{\alpha}_\infty)$ be a unital and equivariant $*$-homomor\-phism. Consider a lift $(\kappa_n)_n: B\to\ell^\infty(\IN,A)$, where each $\kappa_n$ is a $*$-linear map. Then this sequence satisfies 
\begin{itemize}
\item $\dst\limsup_{n\to\infty}\|\kappa_n(b)a\| \leq\|b\|$
\item $\dst\limsup_{n\to\infty}\|\big( \kappa_n(b_1)\kappa_n(b_2)-\kappa_n(b_1b_2) \big)a\| = 0$
\item $\dst\limsup_{n\to\infty}\|\kappa_n(\eins_B)a-a\|= 0$
\item $\dst\limsup_{n\to\infty}\|[\kappa_n(b),a]\|= 0$
\item $\dst\limsup_{n\to\infty}\| \big( (\alpha_g\circ\kappa_n)(b)-(\kappa_n\circ\beta_g)(b) \big)a \| = 0$ 
\item $\dst\limsup_{n\to\infty}~\max_{h\in K}\Big( \|\beta_h(b)-b\|- \| \big( (\alpha_h\circ\kappa_n)(b)-(\kappa_n)(b) \big)a \| \Big) \geq 0$
\end{itemize}
for all $b,b_1,b_2\in B$, $0\leq a\leq\eins$ in $A$, $g\in G$ and all compact sets $K\subset G$. Note that the last condition follows from \ref{uniform-equivariance}.

Apply \ref{scc-approx-conjugate} and find a sequence of automorphisms $\phi_n: B\to B$ with 
\[
\max_{g\in K}\|\beta_g\circ\phi_n-\phi_n\circ\gamma_g\|\to 0
\] 
for all compact $K\subset G$. Let $1_G\in M_n\fin G$ be an increasing sequence of finite sets whose union $G_d$ is a countable, dense subgroup of $G$. Let $0\in F_n\fin B$ be a sequence of finite subsets whose union is a countable, dense $\IQ[i]$-$*$-subalgebra $B'$ of $B$, which is invariant under $\gamma_g$ for every $g\in G_d$. Moreover, let $E_n\fin A_{1,+}$ be an increasing sequence of finite sets whose union is dense. Let $K_0\subset G$ be a compact neighbourhood of $1_G$.

For every $n\in\IN$, we find a number $k_n\in\IN$ satisfying
\begin{itemize}
\item $\|\kappa_{k_n}(b)a\| \leq\|b\|+1/n$
\item $\|\big( \kappa_{k_n}(b_1)\kappa_{k_n}(b_2)-\kappa_{k_n}(b_1b_2) \big)a\| \leq 1/n$
\item $\|\kappa_{k_n}(\eins_B)a-a\|\leq 1/n$
\item $\|[\kappa_{k_n}(b),a]\|\leq 1/n$
\item $\| \big( (\alpha_g\circ\kappa_{k_n})(b)-(\kappa_{k_n}\circ\beta_g)(b) \big)a \| \leq 1/n$
\item $\|\big( (\alpha_h\circ\kappa_n)(b)-\kappa_n(b) \big)a\|\leq \|\beta_h(b)-b\|+1/n$
\end{itemize}
for all $b_1,b_2\in \phi_n(F_n)$, $a\in E_n$, $h\in K_0$, $g\in M_n$ and 
\[
 b\in \set{ \beta_g(\phi_n(x))-\phi_n(\gamma_g(y)) ~|~ x,y\in F_n~\text{and}~g\in M_n }.
\]
Now the last condition, together with the choice of $\phi_n$, implies that for sufficiently large $n$, we have
\[
\begin{array}{cl}
\multicolumn{2}{l}{ \hspace{-15mm}\|\big( (\alpha_h\circ\kappa_{k_n}\circ\phi_n)(b) - (\kappa_{k_n}\circ\phi_n)(b) \big) a\| } \\
\leq&  1/n+\|(\beta_h\circ\phi_n)(b)-(\phi_n)(b)\| \\
\leq& 2/n+\|\gamma_h(b)-b\|
\end{array}
\]
for all $b\in F_n$, $a\in E_n$ and $h\in K_0$. Moreover, the first and fifth condition combined with the choice of the $\phi_n$ also implies
\[
\begin{array}{cl}
\multicolumn{2}{l}{ \hspace{-10mm}\dst \limsup_{n\to\infty} \|\big( (\alpha_g\circ\kappa_{k_n}\circ\phi_n)(b) - (\kappa_{k_n}\circ\phi_n\circ\gamma_g)(b) \big) a\| } \\
\leq&\dst \limsup_{n\to\infty}~\Big( \|\big( (\alpha_g\circ\kappa_{k_n}\circ\phi_n)(b) - (\kappa_{k_n}\circ\beta_g\circ\phi_n)(b) \big) a\| \\
& +  \|\big( (\kappa_{k_n}\circ\beta_g\circ\phi_n)(b) - (\kappa_{k_n}\circ\phi_n\circ\gamma_g)(b) \big) a\|  \Big) \\
\leq& \dst 0+\limsup_{n\to\infty} \|\kappa_{k_n}\big( (\beta_g\circ\phi_n)(b) - (\phi_n\circ\gamma_g)(b) \big) a\| = 0
\end{array} 
\]
for every $g\in M_n$ and $b\in F_n$.

Combining all these observations, we see that the chosen sequence $k_n\in\IN$ satisfies
\begin{itemize}
\item $\dst\limsup_{n\to\infty} \|(\kappa_{k_n}\circ\phi_n)(b)a\| \leq\|b\|$
\item $\|\big( (\kappa_{k_n}\circ\phi_n)(b_1)(\kappa_{k_n}\circ\phi_n)(b_2)-(\kappa_{k_n}\circ\phi_n)(b_1b_2) \big)a\| \to 0$
\item $\|(\kappa_{k_n}\circ\phi_n)(\eins_B)a-a\| \to 0$
\item $\|[(\kappa_{k_n}\circ\phi_n)(b),a]\| \to 0$
\item $\| \big( (\alpha_g\circ\kappa_{k_n}\circ\phi_n)(b)-(\kappa_{k_n}\circ\phi_n\circ\gamma_g)(b) \big)a \| \to 0$ 
\item $\dst\limsup_{n\to\infty}~\max_{h\in K_0} \Big( \|\gamma_h(b)-b\| - \|\big( (\alpha_h\circ\kappa_{k_n}\circ\phi_n)(b)-(\kappa_{k_n}\circ\phi_n)(b) \big)a\| \Big)\geq 0$
\end{itemize}
for all $b,b_1,b_2\in B'$, $a\in\bigcup_n E_n$ and $g\in G_d$.

Thus we get a well-defined, unital and continuous $*$-homomorphism 
\[
\psi: B'\to F_\infty(A) \quad\text{via}\quad \psi(b)=[ \big( (\kappa_{k_n}\circ\phi_n)(b) \big)_n ].
\] 
The last condition above implies that 
\[
\|\tilde{\alpha}_{\infty,h}(\psi(b))-\psi(b)\| \leq \|\gamma_h(b)-b\|
\]
for all $b\in B'$ and $h\in K_0$. As $K_0$ is a neighbourhood of $1_G$, it follows that the image of $\psi$ is contained in the continuous part $F_{\infty,\alpha}(A)$.
 
Moreover, the fifth condition implies that we have $\alpha_{\infty,g}\circ\psi=\psi\circ\gamma_g$ for all $g\in G_d$. By continuity, $\psi$ uniquely extends to a unital $*$-homomorphism $\psi: B\to F_{\infty,\alpha}(A)$, and the equation $\alpha_{\infty,g}\circ\psi=\psi\circ\gamma_g$ must now hold for every $g\in G$ because of continuity and the fact that $G_d\subset G$ is dense. This finishes the proof.
\end{proof}

With these intermediate technical results, we can now get a more intrinsic characterization of semi-strongly self-absorbing \cstar-dynamical systems:

\begin{theorem}
\label{sssa-ssa}
Let $G$ be a second-countable, locally compact group and $\CD$ a separable, unital \cstar-algebra.
For an action $\beta: G\curvearrowright\CD$, the following are equivalent:
\begin{enumerate}[label=\textup{(\roman*)},leftmargin=*]
\item $\beta$ is semi-strongly self-absorbing. \label{sssa-ssa:1}
\item $\beta$ has approximately $G$-inner half-flip and there exists a unital, equivariant $*$-homomorphism from $(\CD,\beta)$ to $\big( \CD_{\infty,\beta}\cap \CD', \beta_\infty \big)$. \label{sssa-ssa:2}
\item $\beta$ has approximately $G$-inner half-flip and is strongly cocycle conjugate to $( \bigotimes_\IN \CD, \bigotimes_\IN \beta)$. \label{sssa-ssa:3}
\end{enumerate}
\end{theorem}
\begin{proof}
\ref{sssa-ssa:1}$\implies$\ref{sssa-ssa:2}: Let $\gamma: G\curvearrowright\CD$ be a strongly self-absorbing action that is strongly cocycle conjugate to $\beta$. By \ref{ssa-facts}\ref{ssa-facts:2} and \ref{scc-aih}, it follows that $\beta$ has approximately $G$-inner half-flip.
As we have $\beta\otimes\gamma\scc\gamma\otimes\gamma\scc\gamma\scc\beta$, it follows from \ref{ssa-facts}\ref{ssa-facts:1} that there exists a unital and equivariant $*$-homomorphism from $(\CD,\gamma)$ to $\big( \CD_{\infty,\beta}\cap \CD', \beta_\infty \big)$. By \ref{scc-homo-ex}, there also exists a unital and equivariant $*$-homomorphism from $(\CD,\beta)$ to $\big( \CD_{\infty,\beta}\cap \CD', \beta_\infty \big)$.

\ref{sssa-ssa:2}$\implies$\ref{sssa-ssa:3}: Applying \ref{aih-ssa}, we see that the infinite tensor power \linebreak $( \bigotimes_\IN \CD, \bigotimes_\IN \beta)$ is strongly self-absorbing. Since there exists a unital and equivariant $*$-homomorphism from $(\CD,\beta)$ to $\big( \CD_{\infty,\beta}\cap \CD', \beta_\infty \big)$, one can construct a unital and equivariant $*$-homomorphism $(\bigotimes_\IN \CD, \bigotimes_\IN \beta)$ to $\big( \CD_{\infty,\beta}\cap \CD', \beta_\infty \big)$, by applying a standard reindexation argument (see also \cite[1.13]{Kirchberg04}) and using \ref{uniform-equivariance} in the process. By \ref{equi-McDuff}, it follows that
\[
(\CD,\beta) \scc \big( \CD\otimes\bigotimes_\IN\CD, \beta\otimes\bigotimes_\IN\beta \big) = ( \bigotimes_\IN \CD, \bigotimes_\IN \beta).
\]

\ref{sssa-ssa:3}$\implies$\ref{sssa-ssa:1}: This follows directly from \ref{aih-ssa}\ref{aih-ssa:2}.
\end{proof}

We may in fact extend our equivariant McDuff-type main result \ref{equi-McDuff} to cover semi-strongly self-absorbing actions:

\begin{theorem}
\label{equi-McDuff-2}
Let $G$ be a second-countable, locally compact group. Let $A$ be a separable \cstar-algebra and $(\alpha,u): G\curvearrowright A$ a cocycle action. Let $\CD$ be a separable, unital \cstar-algebra and $\beta: G\curvearrowright\CD$ an action such that $(\CD,\beta)$ is semi-strongly self-absorbing.
The following are equivalent:
\begin{enumerate}[label=\textup{(\roman*)},leftmargin=*] 
\item $(A,\alpha,u)$ is strongly cocycle conjugate to $(A\otimes\CD,\alpha\otimes\beta,u\otimes\eins_\CD)$. 
\item $(A,\alpha,u)$ is cocycle conjugate to $(A\otimes\CD,\alpha\otimes\beta,u\otimes\eins_\CD)$. 
\item There exists a unital, equivariant $*$-homomorphism from $(\CD,\beta)$ to $\big( F_{\infty,\alpha}(A), \tilde{\alpha}_\infty \big)$.
\end{enumerate}
\end{theorem}
\begin{proof}
This follows directly from \ref{equi-McDuff} and \ref{scc-homo-ex}.
\end{proof}

We conclude this section by considering the special case where the acting group is compact.

\begin{prop}
\label{scc-compact}
Let $G$ be a second-countable, compact group. Let $A$ be a separable, unital \cstar-algebra and $\alpha,\beta: G\curvearrowright A$ two actions. Then $\alpha$ and $\beta$ are strongly cocycle conjugate if and only if they are conjugate.
\end{prop}
\begin{proof}
Let $\phi: A\to B$ be an isomorphism and $x_n\in\CU(A)$ a sequence of unitaries such that $w_g=\lim_{n\to\infty} x_n\beta_g(x_n^*)$ converges uniformly and $\ad(w_g)\circ\beta_g\circ\phi=\phi\circ\alpha_g$ for all $g\in G$. For some large $n$, we may replace $\phi$ by $\ad(x_n^*)\circ\phi$ and $w_g$ by $x_n^*w_g\beta_g(x_n)$ (see the proof of \ref{scc-approx-conjugate}), and assume
\[
\ad(w_g)\circ\beta_g\circ\phi=\phi\circ\alpha_g\quad\text{for all}~g\in G
\]
and $\max_{g\in G} \|\eins-w_g\|<1$. Applying \cite[2.4]{Izumi04}, we deduce that $[g\mapsto w_g]$ is a coboundary, i.e.~there exists $v\in\CU(A)$ with $w_g=v\beta_g(v^*)$ for all $g\in G$. This implies
\[
\beta_g\circ\ad(v^*)\circ\phi=\ad(v^*)\circ\phi\circ\alpha_g\quad\text{for all}~g\in G.
\]
So $\ad(v^*)\circ\phi: (A,\alpha)\to (A,\beta)$ is an equivariant isomorphism.
\end{proof}

\begin{cor}
Let $G$ be a second-countable, compact group. Let $\CD$ be a separable, unital \cstar-algebra and $\gamma: G\curvearrowright\CD$ an action. Then $\gamma$ is strongly self-absorbing if and only if $\gamma$ is semi-strongly self-absorbing.
\end{cor}

\begin{cor}
If one restricts to compact $G$ and genuine actions, then all instances of 'strongly cocycle conjugate' in \ref{equi mcduff} can be replaced by 'conjugate'.
\end{cor}

\begin{question}
Does \ref{scc-compact} hold also in the case that $A$ is non-unital? Does \ref{scc-compact} hold, if $(\alpha,u), (\beta,w): G\curvearrowright A$ are cocycle actions with non-trivial cocycles?
\end{question}


\section{Examples}
\noindent
So far, we have only discussed a very straightforward class of examples of strongly self-absorbing actions, namely the trivial actions on strongly self-absorbing \cstar-algebras. In this section, we shall discuss other examples of (semi-)strongly self-absorbing actions. Let us start by considering compact group actions on strongly self-absorbing \cstar-algebras that have the Rokhlin property, see \cite[3.2]{HirshbergWinter07} for the definition on unital \cstar-algebras.

\begin{example}
\label{Rokhlin-prop}
Let $G$ be a second-countable, compact group and $\CD$ a strongly self-absorbing \cstar-algebra. Let $\gamma: G\curvearrowright\CD$ be an action with the Rokhlin property. Then $(\CD,\gamma)$ is strongly self-absorbing.
\end{example}
\begin{proof}
Note that all unital $*$-homomorphism of the form $\CD\to\CD\otimes A$ are mutually approximately unitarily equivalent by \cite[1.12]{TomsWinter07}. In particular, this is the case for $A=\CC(G)$. 
By \cite[Theorem 5.10]{BarlakSzaboVoigt17} (see \cite[3.5]{Izumi04} for the finite group case), the systems $(\CD,\gamma)$ and $(\CD\otimes\CD,\gamma\otimes\gamma)$ are conjugate. 
Let $\phi: (\CD,\gamma)\to (\CD\otimes\CD,\gamma\otimes\gamma)$ be an equivariant isomorphism. Then $\phi$ is approximately unitarily equivalent to $\id_\CD\otimes\eins_\CD$ as an ordinary $*$-homomorphism. Once again because of the Rokhlin property \cite[Corollary 5.9]{BarlakSzaboVoigt17} (see \cite[Proposition 3.1]{GardellaSantiago15_2} for the finite group case), it follows that these two equivariant $*$-homomorphisms are indeed approximately $G$-unitarily equivalent.
\end{proof}

For the next example, recall Kishimoto's notion of the Rokhlin property for flows \cite{Kishimoto96_R}:

\begin{example}[cf.~\cite{Kishimoto02}]
A quasi-free Rokhlin flow on $\CO_2$ is semi-strongly self-absorbing. In particular, if $s_1,s_2\in\CO_2$ are the canonical generators and $\lambda>0$ is an irrational number, then the flow given by $\gamma_t(s_1)=e^{it}s_1$ and $\gamma_t(s_2)=e^{-i\lambda t}s_2$ is semi-strongly self-absorbing.
\end{example}
\begin{proof}
Let $\gamma: \IR\curvearrowright\CO_2$ be a quasi-free flow with the Rokhlin property. By \cite[3.5]{Kishimoto02}, $\gamma$ has approximately $G$-inner flip. By \cite[3.6]{Kishimoto02}, there exists a unital and equivariant $*$-homomorphism from $(\CO_2,\gamma)$ to $\big( (\CO_2)_\infty\cap\CO_2', \gamma_\infty \big)$. By \ref{sssa-ssa}, it follows that $\gamma$ is semi-strongly self-absorbing.
\end{proof}

Let us now consider examples of discrete group actions that come from noncommutative Bernoulli shifts:

\begin{example}
\label{ex:tensor-flip}
Let $G$ be a finite group. Let $n\in\IN$ be a number and $\sigma: G\curvearrowright\set{1,\dots,n}$ an action. Let $\fp$ be a supernatural number of infinite type. Then the induced tensorial shift action $\beta^\sigma: G\curvearrowright M_\fp^{\otimes n}$ given by
\[
\beta^\sigma_g(a_1\otimes\dots\otimes a_n) = a_{\sigma_g(1)}\otimes\dots\otimes a_{\sigma_g(n)}
\]
is strongly self-absorbing.
\end{example}
\begin{proof}
We first note that since $\fp$ is of infinite type, the UHF algebra $M_\fp$ is strongly self-absorbing. So choosing an isomorphism $M_\fp\cong M_\fp^{\otimes\infty}$, a straightforward rearrangement of the tensorial factors yields that $(M_\fp^{\otimes n},\beta^\sigma) \cong \bigotimes_\IN (M_\fp^{\otimes n},\beta^\sigma)$. By \ref{aih-ssa}, it suffices to show that $\beta^\sigma$ has approximately $G$-inner flip. Now by definition, $(M_\fp^{\otimes n},\beta^\sigma)$ is an equivariant inductive limit of \cstar-dynamical systems of the form $(M_p^{\otimes n}, \beta^\sigma)$ for some natural numbers $p\in\IN$. So without loss of generality, it suffices to show that the tensorial shift action induced by $\sigma$ on $M_p^{\otimes n}$ has $G$-inner flip.

If $\set{e_{i,j}}_{1\leq i,j\leq p}\subset M_p$ are the standard matrix units, then
\[
\set{ e_{\mfi,\mfj}= e_{i_1,j_1}\otimes\dots\otimes e_{i_n,j_n} ~|~ \mfi=(i_1,\dots,i_n), \mfj=(j_1,\dots,j_n)\in\set{1,\dots,p}^n } 
\]
defines another set of matrix units for $M_p^{\otimes n}$. Consider the unitary in $M_p^{\otimes n}\otimes M_p^{\otimes n}$ given by
\[
v=\sum_{\mfi,\mfj\in\set{1,\dots,p}^n} e_{\mfi,\mfj}\otimes e_{\mfj,\mfi}.
\]
Then $v$ implements the flip automorphism. It is clear that $\beta^\sigma$ restricts to a $G$-action on the matrix units $e_{\mfi,\mfj}$, just by acting on the index set. In particular, we see that $v$ is fixed by the diagonal action $\beta^\sigma\otimes\beta^\sigma$ on $M_p^{\otimes n}\otimes M_p^{\otimes n}$. So indeed, $(M_p^{\otimes n},\beta^\sigma)$ has $G$-inner flip. This finishes the proof.
\end{proof}

\begin{question}
Is the action from the above example still strongly self-absorbing, if one replaces $M_\fp$ by some other strongly self-absorbing \cstar-algebra $\CD$? In particular, what about $\CD=\CZ$ or $\CD=\CO_\infty$?
\end{question}

The case $\CD=\CO_2$ indeed has an affirmative answer because a faithful, tensorial shift action of a finite group on $\CO_2$ has the Rokhlin property by \cite[5.4]{Izumi04}, thus falling within the scope of \ref{Rokhlin-prop}.

More generally, let us ask:

\begin{question}
Let $G$ be a countable, discrete group and $\sigma: G\curvearrowright\IN$ an action. Let $\CD$ be a strongly self-absorbing \cstar-algebra. Consider the induced tensorial shift action $\beta^\sigma: G\curvearrowright \bigotimes_\IN\CD\cong\CD$ given by
\[
\beta^\sigma_g(a_1\otimes a_2\otimes a_3\otimes\dots) = a_{\sigma_g(1)}\otimes a_{\sigma_g(2)}\otimes a_{\sigma_g(3)}\otimes\dots ,
\]
where all but finitely many of the $a_n\in\CD$ are equal to $\eins$. When is $(\CD,\beta^\sigma)$ strongly self-absorbing? 
\end{question}

At least a straightforward rearrangement of the tensorial factors can be done again in this case to deduce that one always has $(\bigotimes_\IN\CD,\beta^\sigma)=\bigotimes_\IN (\bigotimes_\IN\CD,\beta^\sigma)$. So in view of \ref{aih-ssa}, the question is in which cases $\beta^\sigma$ has approximately $G$-inner half-flip.

The following trivial argument allows one to construct strongly self-absorbing actions of residually compact groups, with the help of examples of compact group actions.

\begin{prop}
Let $G$ be a second-countable, locally compact group. Let $H\subset G$ be a closed, normal, cocompact subgroup. Let $\pi: G\to G/H$ be the quotient map. Let $\CD$ be a separable, unital \cstar-algebra and $\gamma: G/H\curvearrowright\CD$ a strongly self-absorbing $G/H$-action. Then $\gamma\circ\pi: G\curvearrowright\CD$ is strongly self-absorbing.
\end{prop}

\begin{rem}
\label{ex:ssa-cocompact-products}
In particular, let $H_n\subset G$ be a sequence of closed, normal, cocompact subgroups with $\bigcap_{n\in\IN} H_n = \set{1_G}$. Let $\pi_n: G\to G/H_n$ be the quotient maps, and let $\gamma^{(n)}: G/H_n\curvearrowright\CD_n$ be a sequence of strongly self-absorbing actions on some separable, unital \cstar-algebras. Then
\[
\gamma = \bigotimes_{n\in\IN} (\gamma^{(n)}\circ\pi_n): G\curvearrowright \bigotimes_{n\in\IN}\CD_n
\] 
is a strongly self-absorbing action. If each $\gamma^{(n)}$ was faithful, then $\gamma$ is faithful.
\end{rem}

\begin{cor}
\label{ex:res-fin-on-UHF}
Let $\fp$ be a supernatural number of infinite type. Then any countable, discrete, residually finite group admits a faithful, strongly self-absorbing action on $M_{\fp}$.
\end{cor}
\begin{proof}
Let $G$ be a countable, discrete, residually finite group. Let $H_n\subset G$ be a sequence of normal subgroups with finite index and $\bigcap_{n\in\IN} H_n=\set{1_G}$. By \ref{ex:tensor-flip}, the tensorial flip action of $G/H_n$ on $M_\fp$, which is induced by the left translation action $G/H_n\curvearrowright G/H_n$, is faithful and strongly self-absorbing. By \ref{ex:ssa-cocompact-products}, it follows that the infinite tensor product of the induced $G$-actions is a faithful, strongly self-absorbing action on $M_\fp$. 
\end{proof}

\begin{theorem}
Let $\fp$ be a supernatural number of infinite type. Every pointwise strongly outer action of $\IZ^d$ on $M_\fp$ is semi-strongly self-absorbing.
\end{theorem}
\begin{proof}
By \ref{ex:res-fin-on-UHF}, there exists a faithful, strongly self-absorbing $\IZ^d$-action $\gamma$ on $M_\fp$. Because of \ref{ssa-facts}\ref{ssa-facts:1}, it is clearly pointwise strongly outer. By Matui's uniqueness result \cite[5.4]{Matui11}, every other pointwise strongly outer action of $\IZ^d$ on $M_\fp$ is strongly cocycle conjugate to $\gamma$, so the claim follows.
\end{proof}

\begin{example}
Let $G$ be a finite group. Then any quasi-free action of $G$ on $\CO_\infty$ is strongly self-absorbing.
\end{example}
\begin{proof}
Let $\alpha: G\curvearrowright\CO_\infty$ be a quasi-free action. By \cite[Section 2]{Evans80}, every automorphism $\alpha_g$ is either trivial or outer. So by dividing out the kernel of $\alpha$, if necessary, we may assume that $\alpha$ is pointwise outer. Now it follows from Goldstein-Izumi's absorption theorem \cite[5.1]{GoldsteinIzumi11} that $(\CO_\infty,\alpha) \cong \bigotimes_\IN (\CO_\infty,\alpha)$. On the other hand, \cite[4.1]{GoldsteinIzumi11} implies that $\alpha$ has approximately $G$-inner half-flip. With \ref{aih-ssa}, it follows that $\alpha$ is a strongly self-absorbing action.
\end{proof}

With the argument from \ref{ex:ssa-cocompact-products}, we again get more examples of strongly self-absorbing actions on $\CO_\infty$:

\begin{cor}
\label{ex:res-fin-on-Oinf}
Any countable, discrete, residually finite group admits a faithful, strongly self-absorbing action on $\CO_\infty$.
\end{cor}

\begin{theorem}
Let $\CD$ be a strongly self-absorbing Kirchberg algebra satisfying the UCT. (That is, let either $\CD\cong\CO_\infty$, $\CD\cong\CO_2$ or $\CD\cong\CO_\infty\otimes M_\fp$ for a supernatural number $\fp$ of infinite type.)
Then every pointwise outer $\IZ^d$-action on $\CD$ is semi-strongly self-absorbing. 
\end{theorem}
\begin{proof}
By \ref{ex:res-fin-on-Oinf}, there exists a faithful, strongly self-absorbing action $\gamma: \IZ^d\curvearrowright\CD$. This action must then be pointwise outer because of \ref{ssa-facts}\ref{ssa-facts:1}. By invoking either \cite[5.2]{Matui08} (for the case $\CD\cong\CO_2$) or \cite[6.18, 6.20]{IzumiMatui10} (for the other cases), we can deduce that every pointwise outer $\IZ^d$-action on $\CD$ is strongly cocycle conjugate to $\gamma$. This finishes the proof.
\end{proof}

\begin{question}
When is a quasi-free action of a countable, discrete group on $\CO_\infty$ semi-strongly self-absorbing?
\end{question}

\begin{rem}[cf.~{\cite[Section 6]{GoldsteinIzumi11}}]
\label{Oinfty-model-action}
Let $G$ be a countable, exact and discrete group. By Kirchberg-Phillips' $\CO_2$-embedding result, we can embed $C^*_r(G)$ into $\CO_2$ unitally. In particular, there exists a faithful unitary representation $v: G\to\CU(\CO_2)$. Consider the induced infinite tensor product action $\delta: G\curvearrowright\bigotimes_\IN\CO_2$ given by $\delta_g=\bigotimes_\IN\ad(v(g))$ for all $g\in G$.

 Choosing some non-zero (and necessarily non-unital) $*$-homomorphism $\iota: \CO_2\to\CO_\infty$, we also obtain a faithful unitary representation $u: G\to\CU(\CO_\infty)$ via $u(g)=\iota(v(g))+\eins-\iota(\eins)$ for all $g\in G$. We can then also consider the induced infinite tensor product action $\gamma: G\curvearrowright\bigotimes_\IN\CO_\infty$ given by $\gamma_g=\bigotimes_\IN\ad(u(g))$ for all $g\in G$.

Now assume that $G$ is finite. In this case, it was shown by Goldstein and Izumi in the proof of \cite[6.2]{GoldsteinIzumi11} that the action $\gamma$ is conjugate to any faithful, quasi-free $G$-action on $\CO_\infty$. In particular, it follows that $\gamma$ is strongly self-absorbing. On the other hand, it is not hard to construct a unital $*$-homomorphism from $\CO_2$ into the fixed point algebra $\big( (\CO_2)_\infty\cap\CO_2' \big)^{\delta_\infty}$, so it follows from \ref{equi mcduff} and \ref{scc-compact} that $(\CO_2,\delta)\cong (\CO_2\otimes\CO_2,\delta\otimes\id_{\CO_2})$. Since $\delta$ is pointwise outer, it follows from \cite[4.2]{Izumi04} that $\delta$ has the Rokhlin property. (We note that this can also be seen more directly by using a unital and equivariant inclusion $(\CC(G),G\text{-shift})\into (\CO_2,\ad(v))$.) By \ref{Rokhlin-prop}, it follows that $\delta$ is strongly self-absorbing. Moreover, the actions $\gamma$ and $\delta$ turn out to fit into an equivariant Kirchberg-Phillips absorption result:
\end{rem}

\begin{theorem}[cf.~{\cite[4.3]{Izumi04}} and {\cite[5.1]{GoldsteinIzumi11}}]
Let $G$ be a finite group and $A$ a separable, unital, nuclear, simple \cstar-algebra. Let $\alpha: G\curvearrowright A$ be an action. Then
\begin{enumerate}[label=\textup{(\arabic*)}, leftmargin=*]
\item $(\CO_2\otimes A, \delta\otimes\alpha)\cong (\CO_2,\delta)$.
\item $(A,\alpha)\cong (A\otimes\CO_\infty,\alpha\otimes\gamma)$, if $A$ is purely infinite and $\alpha$ is pointwise outer.
\end{enumerate}
\end{theorem}

The following lists the main results of forthcoming work, which relies on the results and techniques of this paper. It is generalization of the above theorem to the case of countable, amenable groups:

\begin{theorem}[see \cite{Szabo16_K}] 
Let $G$ be countable, discrete, amenable group. 
\begin{enumerate}[label=\textup{(\arabic*)}, leftmargin=*]
\item If $A$ is a Kirchberg algebra and $(\alpha,u): G\curvearrowright A$ is any cocycle action, then $(A,\alpha,u)\scc (A\otimes\CO_\infty, \alpha\otimes\id_{\CO_\infty}, u\otimes\eins_{\CO_\infty})$.
\item The actions $\gamma: G\curvearrowright\CO_\infty$ and $\delta: G\curvearrowright\CO_2$ from \ref{Oinfty-model-action} are strongly self-absorbing.
\item If $A$ is a Kirchberg algebra and $(\alpha,u): G\curvearrowright A$ is a cocycle action such that $\alpha$ is pointwise outer, then $(A,\alpha,u)\scc (A\otimes\CO_\infty, \alpha\otimes\gamma, u\otimes\eins_{\CO_\infty})$.
\item If $A$ is a separable, unital, nuclear, simple \cstar-algebra and $\alpha: G\curvearrowright A$ is an action, then $(\CO_2,\delta)\scc (A\otimes\CO_2,\alpha\otimes\delta)$. Moreover, if $\alpha$ is pointwise outer, then $(A\otimes\CO_2,\alpha\otimes\id_{\CO_2})\scc (\CO_2,\delta)$.
\end{enumerate}
\end{theorem}


\bibliographystyle{gabor}
\bibliography{master}

\end{document}